\documentclass[11pt]{article}

\usepackage{amsmath}
\usepackage{amsthm}
\usepackage{amsfonts}
\usepackage{setspace}
\usepackage{fullpage}
\usepackage{amssymb}
\usepackage{enumitem}
\usepackage{caption}

\usepackage{floatpag}
\usepackage[dvipsnames]{xcolor}

\usepackage[initials]{amsrefs} 

\usepackage{multirow}
\usepackage{float}
\usepackage{tikz} 
\usetikzlibrary{patterns,snakes} 
\usepackage{ifthen}
\usepackage{hyperref}
\hypersetup{colorlinks=true,
citecolor=blue,
filecolor=blue,
linkcolor=blue,
urlcolor=blue
}
\usepackage{cleveref}

\newboolean{colored}
\setboolean{colored}{false} 

\usepackage{comment}

\newcounter{sec}  
\newtheorem{lemma}{Lemma}[sec]

\newtheorem{theorem}{Theorem} 
\newtheorem{corollary}{Corollary}

\newtheorem{claim}{Claim}
\newtheorem{proposition}{Proposition}
\newtheorem{conjecture}{Conjecture}

\newcommand{\C}{\mathcal{C}}

\newcommand{\F}{\operatorname{Forb}}

\newcommand{\s}{\mathcal{S}}
\newcommand{\h}{\mathcal{H}}

\newcommand{\floor}[1]{\left\lfloor{#1}\right\rfloor}
\newcommand{\ceil}[1]{\left\lceil{#1}\right\rceil}

\ifthenelse{\boolean{colored}}{
\newcommand{\rby}{%
	\begin{tikzpicture}[thick, scale=0.05]
	\draw[red] (90:4) -- (210:4);
	\draw[yellow] (90:4)--(-30:4);
	\draw[blue] (210:4) -- (-30:4);
	\end{tikzpicture}
}
\newcommand{\rrr}{%
	\begin{tikzpicture}[thick, scale=0.05]
	\draw[red] (90:4) -- (210:4);
	\draw[red] (90:4)--(-30:4);
	\draw[red] (210:4) -- (-30:4);
	\end{tikzpicture}
}
\newcommand{\yyy}{%
	\begin{tikzpicture}[thick, scale=0.05]
	\draw[yellow] (90:4) -- (210:4);
	\draw[yellow] (90:4)--(-30:4);
	\draw[yellow] (210:4) -- (-30:4);
	\end{tikzpicture}
}
\newcommand{\bbb}{%
	\begin{tikzpicture}[thick, scale=0.05]
	\draw[blue] (90:4) -- (210:4);
	\draw[blue] (90:4)--(-30:4);
	\draw[blue] (210:4) -- (-30:4);
	\end{tikzpicture}
}
 \newcommand{\rrb}{%
	\begin{tikzpicture}[thick, scale=0.05]
	\draw[red] (90:4) -- (210:4);
	\draw[red] (90:4)--(-30:4);
	\draw[blue] (210:4) -- (-30:4);
	\end{tikzpicture}
}
 \newcommand{\rry}{%
	\begin{tikzpicture}[thick, scale=0.05]
	\draw[red] (90:4) -- (210:4);
	\draw[red] (90:4)--(-30:4);
	\draw[yellow] (210:4) -- (-30:4);
	\end{tikzpicture}
}
 \newcommand{\bbr}{%
	\begin{tikzpicture}[thick, scale=0.05]
	\draw[blue] (90:4) -- (210:4);
	\draw[blue] (90:4)--(-30:4);
	\draw[red] (210:4) -- (-30:4);
	\end{tikzpicture}
}
 \newcommand{\bby}{%
	\begin{tikzpicture}[thick, scale=0.05]
	\draw[blue] (90:4) -- (210:4);
	\draw[blue] (90:4)--(-30:4);
	\draw[yellow] (210:4) -- (-30:4);
	\end{tikzpicture}
}
 \newcommand{\yyr}{%
	\begin{tikzpicture}[thick, scale=0.05]
	\draw[yellow] (90:4) -- (210:4);
	\draw[yellow] (90:4)--(-30:4);
	\draw[red] (210:4) -- (-30:4);
	\end{tikzpicture}
}
 \newcommand{\yyb}{%
	\begin{tikzpicture}[thick, scale=0.05]
	\draw[yellow] (90:4) -- (210:4);
	\draw[yellow] (90:4)--(-30:4);
	\draw[blue] (210:4) -- (-30:4);
	\end{tikzpicture}
}
}
{
	
	\newcommand{\rby}{%
	\begin{tikzpicture}[scale=0.05]
	\draw[red, line width=0.55mm, densely dotted] (90:4) -- (210:4);
	\draw[Dandelion, snake, line width=0.5mm, segment length=0.8mm, segment amplitude=0.25mm] (90:4)--(-30:4);
	\draw[blue,line width=0.5mm] (210:4) -- (-30:4);
	\end{tikzpicture}
}
\newcommand{\rrr}{%
	\begin{tikzpicture}[scale=0.05]
	\draw[red, line width=0.55mm, densely dotted] (90:4) -- (210:4);
	\draw[red, line width=0.55mm, densely dotted] (90:4)--(-30:4);
	\draw[red, line width=0.55mm, densely dotted] (210:4) -- (-30:4);
	\end{tikzpicture}
}
\newcommand{\yyy}{%
	\begin{tikzpicture}[scale=0.05]
	\draw[Dandelion, snake, line width=0.5mm, segment length=0.8mm, segment amplitude=0.25mm] (90:4) -- (210:4);
	\draw[Dandelion, snake, line width=0.5mm, segment length=0.8mm, segment amplitude=0.25mm] (90:4)--(-30:4);
	\draw[Dandelion, snake, line width=0.5mm, segment length=0.8mm, segment amplitude=0.25mm] (210:4) -- (-30:4);
	\end{tikzpicture}
}
\newcommand{\bbb}{%
	\begin{tikzpicture}[scale=0.05]
	\draw[blue,line width=0.5mm] (90:4) -- (210:4);
	\draw[blue,line width=0.5mm] (90:4)--(-30:4);
	\draw[blue,line width=0.5mm] (210:4) -- (-30:4);
	\end{tikzpicture}
}
\newcommand{\rrb}{%
	\begin{tikzpicture}[scale=0.05]
	\draw[red, line width=0.55mm, densely dotted] (90:4) -- (210:4);
	\draw[red, line width=0.55mm, densely dotted] (90:4)--(-30:4);
	\draw[blue,line width=0.5mm] (210:4) -- (-30:4);
	\end{tikzpicture}
}
\newcommand{\rry}{%
	\begin{tikzpicture}[scale=0.05]
	\draw[red, line width=0.55mm, densely dotted] (90:4) -- (210:4);
	\draw[red, line width=0.55mm, densely dotted] (90:4)--(-30:4);
	\draw[Dandelion, snake, line width=0.5mm, segment length=0.8mm, segment amplitude=0.25mm] (210:4) -- (-30:4);
	\end{tikzpicture}
}
\newcommand{\bbr}{%
	\begin{tikzpicture}[scale=0.05]
	\draw[blue,line width=0.5mm] (90:4) -- (210:4);
	\draw[blue,line width=0.5mm] (90:4)--(-30:4);
	\draw[red, line width=0.55mm, densely dotted] (210:4) -- (-30:4);
	\end{tikzpicture}
}
\newcommand{\bby}{%
	\begin{tikzpicture}[scale=0.05]
	\draw[blue,line width=0.5mm] (90:4) -- (210:4);
	\draw[blue,line width=0.5mm] (90:4)--(-30:4);
	\draw[Dandelion, snake, line width=0.5mm, segment length=0.8mm, segment amplitude=0.25mm] (210:4) -- (-30:4);
	\end{tikzpicture}
}
\newcommand{\yyr}{%
	\begin{tikzpicture}[scale=0.05]
	\draw[Dandelion, snake, line width=0.5mm, segment length=0.8mm, segment amplitude=0.25mm] (90:4) -- (210:4);
	\draw[Dandelion, snake, line width=0.5mm, segment length=0.8mm, segment amplitude=0.25mm] (90:4)--(-30:4);
	\draw[red, line width=0.55mm, densely dotted] (210:4) -- (-30:4);
	\end{tikzpicture}
}
\newcommand{\yyb}{%
	\begin{tikzpicture}[scale=0.05]
	\draw[Dandelion, snake, line width=0.5mm, segment length=0.8mm, segment amplitude=0.25mm] (90:4) -- (210:4);
	\draw[Dandelion, snake, line width=0.5mm, segment length=0.8mm, segment amplitude=0.25mm] (90:4)--(-30:4);
	\draw[blue,line width=0.5mm] (210:4) -- (-30:4);
	\end{tikzpicture}
}
}

\newcommand{\kseven}{
	\begin{tikzpicture}
		\def \n {7} 
		\def \radius {1cm} 
		\foreach \s in {1,...,\n} { 
			\node[draw, circle, inner sep=1pt] (\s ) at ({360/\n * ( \s -0.25)}:\radius) {}; 
		}
		\draw[blue, line width = 0.5mm] (1)--(2)--(3)--(4)--(5)--(6)--(7)--(1);
		\draw[Dandelion, snake, line width=0.5mm, segment length=0.8mm, segment amplitude=0.25mm](1)--(3)--(5)--(7)--(2)--(4)--(6)--(1);
		\draw[red, line width = 0.5mm, densely dotted](7)--(3)--(6)--(2)--(5)--(1)--(4)--(7); 
	\end{tikzpicture}

}

\date{}

\newcounter{num}

\title{The Erd\H{o}s-Hajnal conjecture for three colors and triangles}
\date{\vspace{-5ex}}
\author{ 
Maria Axenovich
\thanks{
	Karlsruhe Institute of Technology, Karlsruhe, Germany;
	e-mail:
	\mbox{\texttt{maria.aksenovich@kit.edu}}\,
}     
\and
Richard Snyder
\thanks{
	Karlsruhe Institute of Technology, Karlsruhe, Germany;
	e-mail: \mbox{\texttt{rjsnyder23@gmail.com}}\,
}
\and
Lea Weber
\thanks{
	Karlsruhe Institute of Technology, Karlsruhe, Germany;
	e-mail: \mbox{\texttt{lea.weber@kit.edu}}\,
}
}

\begin{document}

\maketitle

\begin{abstract}
\setlength{\parskip}{\medskipamount}
\setlength{\parindent}{0pt}
\noindent
Erd\H{o}s and Szekeres's  quantitative version of Ramsey's theorem asserts that in every $2$-edge-coloring of $K_n$ there is  a monochromatic clique on at least $\frac{1}{2}\log n$ vertices. The famous Erd\H{ o}s-Hajnal conjecture claims that forbidding fixed colorings on subgraphs ensures much larger monochromatic cliques. Specifically, the most general, multi-color version of the conjecture states that for any fixed integers $k, s$ and any $s$-edge-coloring $c$ of $K_k$, there exists $\varepsilon > 0$ such that  in any $s$-edge-coloring of $K_n$ that avoids $c$ there is a clique on at least $n^{\varepsilon}$ vertices, using at most $s-1$ colors. 
The conjecture is open in general, though a few partial results are known.

Here, we focus on quantitative aspects of this conjecture in the case when $k = 3$ and $s = 3$, and when there are several forbidden subgraph colorings. More precisely, for a family $\mathcal{H}$ of triangles, each edge-colored with colors from $\{r, b, y\}$,  $\text{Forb}(n,\mathcal{H})$ denotes the family of  edge-colorings of $K_n$ using colors from $\{r, b, y\}$ and containing none of the colorings from $\h$. Let $h_2(n, \h)$ be the maximum $q$ such that any coloring from $\text{Forb}(n, \h)$  has a 
clique on at least $q$ vertices using at most two colors.    We provide bounds on $h_2(n, \h)$ for all families $\h$ consisting of at most three triangles. For most of them, our bounds are asymptotically tight.  This,  in particular, extends a result of  Fox, Grinshpun, and Pach, who determined $h_2(n, \h)$ for $\h$ consisting of a rainbow triangle. In addition, we prove that  for some $\h$,  $h_2(n, \h)$ corresponds to certain  classical Ramsey numbers,  smallest independence number in graphs of given odd girth, or  some other natural graph theoretic parameters.
\end{abstract}

\section*{Introduction}

Erd\H{o}s and Szekeres's  quantitative version \cite{ES} of Ramsey's theorem \cite{R}  asserts that any complete graph on $n$ vertices (also referred to as a clique) edge-colored with two colors has a monochromatic clique on at least $\frac{1}{2} \log n$ vertices.  When the coloring is assumed to have some additional structure, one can expect to find larger monochromatic cliques. Indeed, the famous Erd\H{o}s-Hajnal conjecture \cite{EH}  asserts  that for any fixed integer $k$ and any two edge-colored  clique $K$ on $k$ vertices,  there is a positive constant $a$ such that  in any two-edge-colored  complete $n$-vertex graph with no copy of  $K$,  there is a monochromatic clique on  at least $n^a$ vertices.  Erd\H{o}s and Hajnal \cite{EH}  proved that when $n^a$ is replaced by $e^{ \epsilon \sqrt{\log n}}$, then the statement holds for $\epsilon$ depending on $K$ only.   For a survey on this conjecture, see Chudnovsky \cite{C}.\\

In addition, Erd\H{o}s and Hajnal stated a multicolored version of this conjecture asserting that 
for any fixed integer $k \geq 3$  and for any fixed $s$-edge-colored clique $K$ on $k$ vertices, for $s\geq 2$,  there is a positive constant $a=a(K)$ such that any $s$-edge-coloring of a clique on $n$ vertices  with no copy of $K$ contains a clique on $\Omega(n^a)$ vertices using at most $s-1$ colors.  \\

Here, we say that a clique $K'$ edge-colored with a coloring $c$ contains a {\it copy of } (or simply contains)  a clique $K$  on a vertex set $\{1, \ldots, k\}$ with an edge-coloring $c'$, if there is a set of $k$ vertices  $\{v_1, \ldots, v_k\}$ in $K'$ and a bijection $\phi: \{1,\ldots, k\} \rightarrow \{v_1, \ldots, v_k\}$ such that $c'(ij)= c(\phi(i)\phi(j))$, for all $i, j$, $1\leq i<j\leq k$.   \\

In this paper, to avoid confusion, we will use the term \emph{pattern} for a forbidden coloring of a subgraph, so the term \emph{coloring} will refer to the coloring of a larger graph in which we forbid certain patterns. We consider the case when the number of colors is three and the forbidden patterns are imposed on triangles, but there could be more than one forbidden pattern.  Specifically, we investigate all sets of at most three patterns. We provide all our results in the Tables \ref{t1}, \ref{t2}, \ref{t3}. One can immediately see from Table 
\ref{t1} that the Erd\H{o}s-Hajnal conjecture holds true in this setting.  We focus on the quantitative version of the conjecture and  provide asymptotic bounds on the sizes of the largest $2$-edge-colored cliques.\\

All of the colorings considered here use colors $r, b$, and $y$, corresponding to `red', `blue', and `yellow'. 
The complete graph on $n$ vertices is denoted by $K_n$.
We call an edge-colored complete graph $K_k$ using at most two colors on its edges a {\it two-colored $k$-clique}.
We also call the set of vertices of a two-colored clique a {\it two-colored set}.
For a family $\mathcal H$ of patterns using colors $r, b, y$,  we define the class of $\mathcal H$-{\it avoiding edge-colorings}  of $K_n$ as a family using colors $r, b, y$, and containing none of the patterns from $\mathcal H$.  We denote the family of all $\mathcal H$-{avoiding edge-colorings}  by $\F(n, \mathcal H)$.\\

For an edge-coloring $c$, let
$$h_2(c) = \max\{k \in \mathbb N \ |\  c ~ \text{ contains a two-colored $k$-clique} \}   \text{, and }$$
$$h_2(n, \mathcal H)=\min\{h_2(c)\ |\ c \in \F(n, \mathcal H)\}.$$

In particular,  each edge-coloring of $K_n$ using three colors  and not containing patterns from $\mathcal H$ contains a clique on $h_2(n, \mathcal H)$ vertices using at most $2$ colors. In addition,  there is an $\mathcal H$-avoiding coloring with every clique on more  than $h_2(n, \mathcal H)$ vertices using all three colors. \\

We consider all sets  $\mathcal H$ of at most three  patterns on triangles using colors from  $\{r, b, y\}$. 
We also write, for example, $rrb$ to represent  a coloring of a triangle  with one  blue and two  red edges. 
These patterns are  $rrr, bbb, yyy,  rrb, rry, bbr, bby, yyr, yyb, rby$.  Note that if for two families $\h, \h'$ there is a permutation $\pi$ of colors, such that $\h'$ is obtained by applying $\pi$ to each pattern in $\h$, we have $h_2(n, \mathcal H) = h_2(n, \mathcal H')$.  If for two sets of avoiding patterns, $\mathcal H $ and  $\mathcal H'$, we have that   $\mathcal H \subseteq \mathcal H'$, then $h_2(n, \mathcal H) \le h_2(n, \mathcal H')$. Indeed, this holds since any  $\mathcal H'$-avoiding coloring is also an $\mathcal H$-avoiding coloring.\\

Two of the entries of our tables are expressed in terms of functions $f(n)$ and $g(n)$ that are interesting in their own right. We denote by $\alpha(G)$ and $\omega(G)$ the \emph{independence number} and the \emph{clique number}, respectively, of a graph $G$ (i.e., the largest integers $\alpha$ and $\omega$ such that $G$ contains a set of $\alpha$ vertices inducing no edges, and a set of $\omega$ vertices inducing a clique).  For a graph $G$,  let $G^2$ be the square of $G$, i.e., the graph on the same vertex set as $G$ with two vertices adjacent if and only if they are at distance at most $2$ in $G$. Let  $$f(G)= \max \{\alpha(G),\omega(G^2)\} ~\text{ and } f(n)= \min \{f(G): ~ |G|=n,~ \omega(G)=2\}.$$ Further, recall that the \emph{odd girth}, $g_{\rm odd}(G)$,  of a graph $G$ is the length of a shortest odd cycle in $G$. 
Let $$g(n)= \min\{\alpha(G):  ~|G|=n, ~ g_{\rm odd}(G) \geq 7\}.$$

\begin{figure}
	\thisfloatpagestyle{empty}
		\captionsetup{justification=centering}
\begin{table}[H]
\begin{center}
	\begin{tabular}{l|l|l} 
		$\mathcal H$ & $h_2(n, \mathcal H)$  &  Results  \\ \hline
		\hline 
		$\{\rby\}$ & $\Theta( n^{1/3}\log^2 n)$ & \cite{FGP}\\ 
		\hline 
		$\{\rrr\}$ & $\Theta(\sqrt{n \log n}) $ & \Cref{lem1.rrr}\\
		\hline 
		$\{\rrb\}$ & $\ceil{\sqrt n}$ & \Cref{lem1.rrb}\\
		\hline 
	\end{tabular}
\end{center}
\caption{Bounds on $h_2(n, \h)$ for families $\h$  of one pattern on a triangle} \label{t1}
\end{table}

\begin{table}[H]
\begin{center}
	\begin{tabular}{l|l|l} 	
		$\mathcal H$ & $h_2(n, \mathcal H)$  &  Results  \\ \hline	
		\hline 
			$\{\rrb, \rry\}, \{\rrb, \bbr\} , \{\rrb,\bby\},  \{\rry, \bby\}  $ & $\ceil{\sqrt n}$ & \Cref{lem2.1} \\
			\hline 
			$\{\rby, \rrb\}$ & $\Theta(\sqrt n)$& \Cref{lem2.2}\\
			\hline 
		$\{ \rrr, \bby\}, \{ \rrr, \bbr\}, \{\rrr, \bbb \}$ & $\Theta(\sqrt {n \log n})$ &  Lemmata \ref{lem2.3}, \ref{lem2.5} \\
		\hline
		\multirow{2}{*} {$\{\rrr, \rrb \}$} & $\Omega(\sqrt{n \log n}), O(\sqrt{n} \log^{3/2} n)$ & \Cref{lem2.4}\\
		& $f(n) \le h_2(n, \h) \le 2 f(n)$ & \Cref{lem2.f1} \\
		\hline 
		$\{\rby, \rrr\}$ &  $\Omega(n^{2/3}\log^{-3/2} n),  O(n^{2/3}\sqrt{\log n})$ & \Cref{lem2.6} \\
		\hline
	\end{tabular} 
\end{center}
\caption{Bounds on $h_2(n, \h)$ for families $\h$  of two patterns on triangles} \label{t2}
\end{table}

\begin{table}[H]
\begin{center}

	\begin{tabular}{l|l|l} 	
		$\mathcal H$ & $h_2(n, \mathcal H)$  &  Results  \\ \hline
		\hline 		
				$\{\rrb, \rry, \bby\}, \{\rrb, \rry, \bbr\}$ & $\ceil{\sqrt n}$ & \Cref{lem3.1} \\ 
				\hline 
		$\{\rrb, \bbr, \yyr\}, \{\rrb, \bby, \yyr\}$ & $\ceil{n/2}$& Lemmata \ref{lem3.2}, \ref{lem3.3} \\
		\hline 
		$\{\rrb, \bbr, \rby\}$ & $\ceil{n/2}+1$ & \Cref{lem3.5} \\
		\hline 
		$\{\rrb, \rry, \rby\}, \{\rry, \bby, \rby\}$ & $\Omega(\sqrt{n}), O(\sqrt{n\log n})$ & Lemmata \ref{lem3.4}, \ref{lem3.6}, \\
		\hline 
		$\{\rrb, \bby, \rby\}$ & $\Theta(n^{2/3})$ & \Cref{lem3.7} \\
		\hline 
		$\{\rrr, \bbb, \yyy \}$ & no $\h$-avoiding col. for $n\ge 17$ & \cite{GG} \\
		\hline 
		$\{\rrr, \bbb, \rrb \}$ & $\Omega(\sqrt{n \log n})$, $O(\sqrt{n} \log^{3/2} n)$ & \Cref{lem3.rrr.bbb.rrb}\\
		\hline 
		$\{\rrr, \bbb, \rry \}$ & $2\floor{n/5} + \epsilon(n)$ & \Cref{lem3.rrr.bbb.rry}\\
		\hline 
		$\{\rrr, \bbb, \yyr \}$ & $\ceil{(n-1)/3}\le h_2(n, \h) \le 2\ceil{n/5}$ & \Cref{lem3.rrr.bbb.yyr} \\
		\hline 
		$\{\rrr, \rrb, \rry\}, \{\rrr, \bbr, \yyr\},\{\rrr, \bby, \yyb \}$ & $\ceil{n/2}$ &  Lemmata \ref{lem3.rrr.rrb.rry}, \ref{lem3.rrr.bbr.yyr}, \ref{lem3.rrr.bby.yyb}  \\
		\hline 
		$\{\rrr, \rrb, \bbr \}$ &  $\Omega(\sqrt{n \log n})$, $O( \sqrt{n} \log^{3/2} n)$ & \Cref{lem3.rrr.rrb.bbr} \\
		\hline 
		$\{\rrr, \rrb, \bby \}$ & $\Omega(\sqrt{n \log n})$,  $O( n^{2/3} \sqrt{\log n})$ & \Cref{lem3.rrr.rrb.bby}\\
		\hline 
		\multirow{2}{*}{$\{\rrr, \rrb, \yyr \}$} & $\Omega(n^{2/3}\log^{1/3}n)$,  $O(n^{3/4}\log n)$ & \Cref{lem3.rrr.rrb.yyr} \\
		 & $g(n) \le h_2(n, \h) \le 2 g(n)$ & \Cref{lem:g(n)} \\ 
		 \hline 
		$\{\rrr, \rrb, \yyb\}$ & $\Omega(n^{2/3} \log^{-1/3} n)$, $O( n^{2/3} \sqrt{\log n})$ & \Cref{lem3.rrr.rrb.yyb} \\
		\hline 
			$\{\rrr, \bbr, \bby\}$ & $\Theta(\sqrt{n\log n})$  & \Cref{lem3.rrr.bbr.bby}\\
		\hline
		$\{\rrr, \bbr, \yyb \}$ & $\ceil{3n/7} + \epsilon_1(n)$ & \Cref{lem3.rrr.bbr.yyb} \\
		\hline 
		$\{\rby, \rrr,\bbb\}$ & $\Omega(n^{3/4}\log^{-3/2} n),  O(n^{3/4}\sqrt{\log n})$ & \Cref{lem3.rby.rrr.bbb} \\
		\hline 
		$\{\rby, \rrr, \bbr \}, \{\rby, \rrr, \bby \}, \{\rby, \rrr, \rrb \} $ & $\Omega(n^{2/3}\log^{-2} n),  O(n^{2/3}\sqrt{\log n})$  & \Cref{lem3.rby.rrr.rest} \\
		\hline
	\end{tabular} 
\end{center}
\caption{Bounds on $h_2(n, \h)$ for families $\h$  of three patterns on triangles \newline 
	$\epsilon(n) = 0$ if $n \equiv 0 \pmod 5$, $\epsilon(n) = 1$ if $n \equiv 1$, and $\epsilon(n) = 2$ otherwise;\newline $\epsilon_1(n) = 1 $ if $n \equiv 2 \pmod{7}$ and $\epsilon_1(n) = 0$, otherwise.
} \label{t3}
\end{table} 
\end{figure}

The paper is structured as follows. 
The Appendix provides a verification that all the sets of patterns have been accounted for. 
We give classical and preliminary results, and more definitions in Section \ref{Preliminaries}. Section \ref{Constructions}
contains most of the constructions we use, and hence yields the upper bounds on $h_2(n, \mathcal{H})$ listed in Tables~\ref{t1}, \ref{t2}, \ref{t3}. The remaining part of the paper provides lemmas and their proofs for the corresponding lower bounds on $h_2(n, \mathcal{H})$. Section \ref{Conclusions} provides final remarks and open questions.

We remark that `$\log$' for us always denotes the base $2$ logarithm. \ifthenelse{\boolean{colored}}{}{We draw red as dotted lines, blue as straight lines and yellow as wavy lines.}

\section{Connections to other results, preliminary results, and more definitions}\label{Preliminaries}

The conclusion of the multicolor Erd\H{o}s-Hajnal conjecture could be restated as: there is a positive constant $a=a(K)$ such that in  any $s$-edge-coloring of a clique on $n$ vertices with no copy of $K$ there is a color class with independence number $\Omega(n^a)$.
Thus, the multicolor Erd\H{o}s-Hajnal conjecture not only extends the respective conjecture for graphs, but puts the problem in the framework of Ramsey problems defined  through some parameter $p$, where one seeks a largest clique colored with a fixed number of colors, such that the parameter $p$ of each color class is bounded by a given number. For example, the classical Ramsey theorems are stated for the parameter $p$ being equal to the clique number, while the Erd\H{o}s-Hajnal conjecture has a formulation as a Ramsey number with parameter $p$, for $p$ equal to the independence number; see other papers \cite{CSh, EP, FKS, M, L, KS, KPPR}, where Ramsey problems with parameter $p$ have been considered for $p$ equal to the diameter, the minimum degree,  the connectivity, and the chromatic number. Fox, Grinshpun, and Pach \cite{FGP} addressed the multi\-colored  Erd\H{o}s-Hajnal  conjecture when $K$ is a rainbow triangle, i.e., a complete graph on $3$ vertices edge-colored using three distinct colors. Among other results, they proved that any such coloring with $s=3$ colors contains a clique using  at most $2$ colors that has order at least $\Omega(n^{1/3} \log^2 n)$. Moreover, they showed that this bound is tight.\\

For a given coloring $c$ in $r, b, y$, we denote by $S_{rb}^c$, $S_{ry}^c$, and $S_{by}^c$ the size of a largest clique using only colors from $\{r,b\}$, only from $
\{r,y\}$, and only from $\{b,y\}$, respectively. 
A pattern is \emph{monochromatic} if only one color occurs, i.e.,  $rrr$, $bbb$ or $yyy$.  
For a subset of colors, e.g., $\{r, b\}$, we say that a graph with all edges colored $r$ or $b$ is a {\it red/blue graph} (or a {\it blue/red graph}). 
If this graph is a clique, we refer to it as a {\it red/blue clique}.
If a pattern is not rainbow, the color used on more than one edge is called the {\it majority color}. For any given color, e.g., red, we say that a vertex $v$  has {\it red  degree}  $k$ if it is incident to exactly $k$ red edges. We denote by $N_r(v)$ the \emph{red neighborhood} of $v$, the set of all vertices joined to $v$ by edges colored red.\\

For positive integers $k$ and $\ell$, we let  $R(k, \ell)$ denote the usual Ramsey number for cliques, i.e., the smallest $N$ such that no matter how the edge set of $K_N$ is colored with red and blue, there is a red $K_k$ or a blue $K_\ell$ in this coloring. For graphs $F$ and $F'$, we let $R(F, F')$ denote the graph Ramsey number of $F$ and $F'$, i.e., the smallest $N$ such that no matter how the edge set of $K_N$ is colored with red and blue, there is a red copy of $F$ or a blue copy of $F'$ in this coloring. Here a copy of $F$ is a graph isomorphic to $F$. This definition easily extends to the case when either $F$ or $F'$ is replaced by a finite family of graphs. 
As usual, we let $V(G), E(G), \chi(G)$, $\alpha(G)$, $\omega(G)$, and $\Delta(G)$ denote the vertex set, edge set, chromatic number, independence number, clique number,  and the maximum degree of a graph $G$, respectively.  For a clique with vertex set $V$, we say that $|V|$ is the {\it size of the clique}, and we often write~``a clique $V$" instead of~``a clique on the vertex set $V$". Most of the colorings we define in \Cref{Constructions} are obtained by `blowing up' other colorings. We make this precise as follows. Suppose $c$ is an edge-coloring of $K_k$ on vertex set $[k]$ and let $V_1, \ldots, V_k$ be nonempty, pairwise disjoint sets. The $(V_1, \ldots, V_k)$\emph{-blowup} of $c$ is the edge-coloring of the complete $k$-partite graph with parts $V_i$ for $i \in [k]$, such that all edges between $V_i$ and $V_j$ have color $c(ij)$ for all $1 \le i < j \le k$.\\

\noindent\textbf{Observation 1:} Let $c_1, c_2$ be 3-edge-colorings of $K_{n_1}, K_{n_2}$. Now let $V_1, \ldots, V_{n_1}$ be vertex disjoint sets of size $|V_i| = n_2$ each and consider the $(V_1, \ldots, V_{n_1})$-blowup of $c_1$, where each $V_i$ is colored according to $c_2$. Let the resulting coloring be $c$.
Then we have $S_{k}^c = S_k^{c_1}\cdot S_k^{c_2} $ for any $k \in \{rb, ry, by\}$.\\

A graph is {\it triangle-free} if it does not contain $K_3$ as a subgraph.  We shall need some results about triangle-free graphs and certain Ramsey numbers.

\begin{theorem}[Kim \cite{K}] \label{thm:K}
For every sufficiently large $n\in \mathbb N$ there exists a triangle-free graph $G$ on $n$ vertices with 
$\alpha(G) \le 9 \sqrt{n \log n}.$
\end{theorem}

The following result gives a corresponding lower bound for the independence number of triangle-free graphs:

\begin{theorem}[Ajtai et al.  \cite{AKS}]\label{thm:AKS}
For any integer $t\geq 3$, we have $R(3,t) = O({t^2}/{\log t})$. That is, there is a constant $C$, such that in any
red/blue edge-coloring of $K_n$ with $n=C{t^2}/{\log t}$  there is either a red $K_3$ or a blue $K_t$.
\end{theorem}

Translating the above Theorem into the language of independent sets in triangle-free graphs, we have:
\begin{corollary}\label{cor:AKS}
For any triangle-free graph $G$ on $n$ vertices, $\alpha(G) = \Omega (\sqrt{n\log n})$.
\end{corollary}

The following result provides an upper bound on the chromatic number of any $n$-vertex triangle-free graph.
\begin{theorem}[Poljak and Tuza \cite{PT}]\label{thm:PT}
For any triangle-free graph $G$ on $n$ vertices, $\chi(G) \le 4 \sqrt{n/\log n}.$
\end{theorem}

We shall also need some known bounds on  Ramsey numbers. We list them here:
$2^{k/2} \leq R(k, k) \leq 4^{k}$, Erd\H{o}s~\cite{E}, Erd\H{o}s and Szekeres~\cite{ES};  
$R(3,k) = \Omega(k^2/{\log k})$, Kim~\cite{K}, Ajtai et al.~\cite{AKS};
$R(4, k) =\Omega(k^{5/2}/\log^2 k)$, Bohman~\cite{B};  $R(C_5, K_k) = O(k^{3/2}/\sqrt{\log k})$, Caro et al.~\cite{CL}; 
$R(3,3,3)= 17$, Greenwood and  Gleason~\cite{GG}; 
$R(\{C_3,C_4,C_5\}, K_k) = \Omega((k/\log k)^{3/4})$,  Spencer~\cite{Sp}. \\

\noindent\textbf{Observation 2:} A lower bound on some Ramsey number $R(s,t)$ corresponds to the existence of a 2-edge-coloring of $K_n$ with no red $K_s$ and no blue $K_t$. 
For example, $R(3,k) = \Omega(k^2/{\log k})$ gives the existence of a red/blue edge-coloring of 
$K_n$ with no red $K_3$ and no blue $K_t$ for any  $t>C \sqrt{n \log n}$ for some constant $C$.\\

A pattern containing all three colors $r$, $b$, and $y$ is called \emph{rainbow}.  Colorings not containing rainbow triangles are called \emph{Gallai colorings}. We will need the following fundamental theorem, which asserts that Gallai colorings have a specific structure.

\begin{theorem}[Gallai \cite{G}]\label{thm:G}
In any Gallai coloring of the complete graph on at least two vertices, the vertex set can be partitioned into at least two non-empty parts such that
\begin{itemize}
    \item for any two distinct parts, all edges between  them are of the same color;
    \item the total number of colors used between parts is at most two.
\end{itemize}

\end{theorem}

\section{Constructions}\label{Constructions}

\newcounter{constr}
\setcounter{constr}{0}

Here we will list some explicit or probabilistic $3$-edge-colorings of $K_n$ which are $\mathcal H$-avoiding for certain families, which we will then refer to in order to prove upper bounds on $h_2(n,\mathcal H)$. We describe the colorings and  bound  the size of a largest $2$-colored set for each of them. The constructions are ordered by increasing order of magnitude of $h_2$. For each constructed coloring $c$ an upper bound on $h_2(c)$ is established. We remark that these upper bounds are  asymptotically tight, but since these facts are not needed for our results, we omit proofs. Some of the constructions are explicit and give exact bounds, while others rely on probabilistic/Ramsey results, and therefore only give asymptotic bounds. Lastly, in each construction where we obtain asymptotic upper bounds, divisibility conditions are ignored, and floors and ceilings are omitted for simplicity of presentation.


\refstepcounter{constr}\label{const1}
\subsection*{Construction \theconstr\   ($\ceil{\sqrt n}$, none of $\rrb, \rry,\bbr, \bby$) }
Let $ n \ge 3$ be an integer, and assume first that $n=m^2$ for some integer $m$. Let the vertex set be $\{v_{ij}:   i, j \in [m]\}$, a set of $m^2$ vertices. Define the coloring $c$ of $E(V)$ as follows:

$$ c(v_{ik}v_{jl}) = \left\{ \begin{array}{cc}
\text{red} & \text{ if } i=j, k \neq l, \\
\text{blue} &  \text{ if } i\neq j, k=l, \\
\text{yellow} &  \text{ if }  i\neq j, k\neq l.
\end{array}\right.$$

The red graph is the disjoint union of $m$ cliques of size $m$, where $m \ge 2$. The same holds for the blue graph. It is not difficult to see that this coloring contains none of the stated patterns.\\

\noindent 
\textbf{Claim:} $h_2(c) \le \ceil{\sqrt n}$.
\begin{proof} Let  $V_i = \{v_{ik} \ |\ k \in [m] \}$, $U_i = \{v_{ji}\ |\ j \in [m] \}$, $i\in [m]$. 
Any blue/yellow clique can use at most one vertex of any given $V_i$, any red/yellow clique can use at most one vertex of any given $U_j$, so we have $S_{by}^c, S_{ry}^c \le m$. In fact, it is not hard to find blue/yellow and red/yellow cliques of size $m$, so $S_{by}^c = S_{ry}^c = m$. Assume there is a red/blue clique of size $m+1$. Then it uses two vertices  $x, y \in V_i$ and one vertex $z \in V_j$ for some $i \neq j$ (here we are using $m \ge 2$). But then $z$ can have a blue edge to at most one of $\{x,y\}$, a contradiction. Thus, we have $S_{rb}^c,  S_{ry}^c, S_{by}^c \le m = \sqrt n$. \end{proof}

\noindent
Now let $n \in \mathbb N$ be arbitrary and take the smallest $m \in \mathbb N$, such that $n \le m^2$. Note that $m = \ceil{\sqrt n}$. Take the construction described above with $m^2$ vertices and arbitrarily remove $m^2 - n$ vertices. Then the size of a largest two-colored clique is still at most $m = \ceil{\sqrt n}$.

\refstepcounter{constr}\label{const:sqrt_n_k4}
\subsection*{Construction \theconstr\ ($O(\sqrt n)$, none of $\rby, \rrb$)}

Let $k = \sqrt{n}$. By Bohman's~\cite{B} result $R(4, t) =\Omega(t^{5/2}/\log^2 t)$ and Observation 2, we may consider a blue/yellow edge-coloring of $K_{k}$ (for sufficiently large $n$) with no yellow clique of size $4$ and no blue clique of size larger than $Ck^{2/5}\log^{4/5}k$, for some constant $C$. Call this coloring $c'$. Now, let $V_1, \ldots, V_{k}$ be pairwise vertex disjoint sets of size $\sqrt{n}$ each and consider the $(V_1, \ldots, V_k)$-blowup of $c'$. Inside each $V_i$ color the edges in red/yellow with no yellow clique of size $4$ and no red clique on more than $Ck^{2/5}\log^{4/5} k$ vertices. Call the resulting coloring $c$. \\
Observe that any triangle in this coloring contains edges from either  one $V_i$ (in which case it is colored in red and yellow), from two distinct $V_i$'s (in which case it is a $bbr, bby, yyr$ or $yyy$ triangle), or from three distinct $V_i$'s (then it uses only colors blue and yellow). Thus, there are no $rby, rrb$ triangles.\\ 

\noindent 
\textbf{Claim:}  $h_2(c) = O(\sqrt n).$
\begin{proof}
Note that any red/blue clique must be the blowup of a blue clique in $c'$ with red cliques inside each $V_i$. By construction, we then have $S_{rb}^c \le (C(\sqrt n)^{2/5}\log^{4/5}(\sqrt n))^2 \le \sqrt n$, for sufficiently large $n$. For red/yellow and blue/yellow cliques we have 
$S_{by}^c, S_{ry}^c \le 3\cdot \sqrt n$.
\end{proof}

\refstepcounter{constr}\label{const4}
\subsection*{Construction \theconstr\  ($O(\sqrt{n\log n})$, none of $ \rrr, \bbr, \bby$) } 

By \Cref{thm:K}, for $n$ large enough there is a triangle-free graph $G$ on $n$ vertices with $\alpha(G) \le 9 \sqrt{n \log n}$. By \Cref{thm:PT}, we have that $\chi(G) \le 4 \sqrt{n/\log n}$. Consider a partition of $V(G)$ into $\chi(G)$ independent sets $V_1,\ldots, V_{\chi(G)}$ (each of size at most $\alpha(G)$). 
Now consider the following $3$-edge-coloring $c$ of $K_n$. Fix a copy of $G$ in $K_n$ and color it red, then color all edges with both endpoints in the same $V_i$ blue and all remaining edges (between two different $V_i$'s that are not in $G$) yellow.\\ 
Observe that the blue graph is a disjoint union of cliques and the red graph is triangle-free, so there are no $rrr, bbr$ or $bby$ triangles.\\

\noindent
\textbf{Claim:} $h_2(c) = O(\sqrt{n \log n})$. 
\begin{proof}
A blue/yellow clique in this coloring corresponds to an independent set in $G$, i.e., we have 
$S_{by}^c = \alpha(G) \le 9 \sqrt{n \log n}$. Since a red/yellow clique contains at most one vertex from any $V_i$, we have $S_{ry}^c \le \chi(G)  \le 4 \sqrt{n/\log n} $. A red/blue clique contains vertices from at most two  $V_i$'s, since otherwise there is a red triangle in $G$, i.e., we have $S_{rb}^c \le 2 \alpha(G) \le 18 \sqrt{n \log n}$.
\end{proof}

\refstepcounter{constr}\label{const5}
\subsection*{Construction \theconstr\ ($O(\sqrt{n\log n})$, none of $\rby, \rrb, \rry $) } 

Take $k=\sqrt{n \log n}$ and consider a blue/yellow edge-coloring $c'$ of $K_k$  without a monochromatic clique of size more than $2\log k$, which exists by the bound on the Ramsey number $R(t,t) \ge 2^{t/2}$. Let $V_1, \ldots, V_k$ be pairwise vertex disjoint sets each of size $n/k = \sqrt{n / \log n} $, and consider the $(V_1, \ldots, V_k)$-blowup of $c'$. Color every edge with both endpoints in $V_i$ red for each $i \in [k]$. Let us denote by $c$ the resulting coloring.

Observe that each triangle in $c$ is either monochromatic red with all vertices in $V_i$, $i\in [k]$,  a $bbr$ or $yyr$ triangle (one edge in a red clique and two edges of the same color to another clique), or 
one of $bbb$, $bby$, $yyy$, $yyb$ (vertices from three different red cliques).\\

\noindent
\textbf{Claim:} $h_2(c) = O(\sqrt{ n\log n})$.
\begin{proof}
Any blue/yellow clique contains at most one vertex from each $V_i$, i.e., we have $S_{by}^c = k =  \sqrt{n \log n} $. Any red/blue (red/yellow) clique in $c$ corresponds to a blue (yellow) clique in $c'$, so we have 
$S_{rb}^c = S_{ry}^c \le (n/k) \cdot 2\log k = \sqrt{n/\log n} \cdot 2\log(\sqrt{n \log n})  \le  2\sqrt{n\log n}. $ \end{proof}

\refstepcounter{constr}\label{const6}
\subsection*{Construction \theconstr\ ($O(\sqrt{n\log n})$, none of $\rby, \rry, \bby )$}
Let $k = \sqrt{n \log n}$. Consider a red/blue edge-coloring $c'$ of $K_k$ without a monochromatic clique of size larger than $2\log k$. Such a coloring exists by the bound on the Ramsey number $R(t,t) \ge 2^{t/2}$. 
Take $m = n/k$ pairwise vertex-disjoint copies of $K_k$ each colored according to $c'$ with vertex sets $V_1, \ldots, V_m$. Color every edge between $V_i$ and $V_j$ yellow for all distinct $i, j \in [m]$. Call the resulting coloring $c$. \\
Observe that each triangle in $c$ is either colored in red and blue (all vertices are in the same $V_i$), or contains at least two yellow edges (vertices in at least two different $V_i$'s).\\

\noindent
\textbf{Claim:} $h_2(c) = O(\sqrt{n \log n} )$.
\begin{proof}
Any red/blue clique contains vertices from at most one $V_i$, so $S_{rb}^c = k = \sqrt{n \log n}$. Any red/yellow (blue/yellow) clique contains at most $2\log k$ vertices from each $V_i$, so we have $S_{ry}^c,S_{by}^c \le (n/k) \cdot 2\log k =\sqrt{n/\log n}\cdot\log(\sqrt{n \log n}) \le 2\sqrt{n \log n} . $
\end{proof}

\refstepcounter{constr}\label{const:r-b-only}
\subsection*{Construction \theconstr\ ($O(\sqrt n \log^{3/2} n)$, none of $\rrr,  \bbb, \bbr, \rrb$)}

From the lower bound on $R(3, t)$, consider a triangle-free graph $H$ with vertex set $[n]$ and independence number at most $9\sqrt{n \log n}$. 
Define $c$, a coloring of the edges of a complete graph on vertex set $[n]$ as follows:  the edges not in $H$ are colored yellow, and each  edge of $H$ is colored red with probability $1/2$ or blue with probability $1/2$.  Note that in this coloring each triangle has a yellow edge, and therefore contains none of the stated patterns. \\

\noindent 
\textbf{Claim:}  With positive probability $h_2(c) = O(\sqrt{n} \log^{3/2} n).$
\begin{proof}
Letting $q(n) = 80 \sqrt{n} \log^{3/2} n$, we shall show that any set of $q(n)$ vertices induces edges of all three colors. 
Let $X$ be a fixed set of $q(n)$ vertices. \\
Now we use Tur\'an's theorem \cite{T}, which tells us that a graph $G$ on $n$ vertices with $\alpha(G) < r$ has at least $\frac{1}{r}\binom{n}{2}$ edges.   
Since $\alpha(G[X]) < 10\sqrt{n \log n}$, the number of edges induced by $X$ in $H$ is at least 
$$ e_X = \frac{1}{10 \sqrt{n \log n}} \binom{q(n)}{2} \ge \frac{q^2(n)}{4\cdot 10 \sqrt{n \log n}}.$$
%
Then the probability that $X$ induces only yellow and blue edges in $c$ or that $X$ induces only yellow and red edges in $c$ is at  most 
$$p_X \leq 2\cdot 2^{-e_X} \leq 2\cdot 2^{-q^{2}(n)/40\sqrt{n \log n}}.$$
Using the union bound over all $q(n)$-element subsets of $[n]$, we have that the probability  that  $c$ contains a $q(n)$-vertex set inducing edges of only two colors is at most
\[
 \binom{n}{q(n)} p_X \leq n^{q(n)} 2^{1 - {q^2(n)}/{40\sqrt{n\log n}}} = 2^{\left(q(n)\log n + 1 - {q^2(n)}/{40\sqrt{n\log n}}\right)} < 1,
 \]
using the definition of $q(n)$.
Thus there is a desired coloring with positive probability (we remark that we did not attempt to optimize the constants here).
\end{proof}

\refstepcounter{constr}\label{const8}
\subsection*{Construction \theconstr\ ($O(n^{2/3})$, none of $\rrb, \rry, \bby, \rby$)}

Consider $n^{1/3}$ pairwise vertex-disjoint red cliques of size $n^{1/3}$ each and color all edges in-between blue. This is a red/blue edge-coloring of $K_{n^{2/3}}$. Consider $n^{1/3}$ pairwise vertex-disjoint copies of $K_{n^{2/3}}$ each colored as above, and color all remaining edges yellow. Call the resulting coloring $c$.\\
Observe that the red graph is a disjoint union of cliques, so any triangle contains either 3, 1 or 0 red edges, if it has its vertices in 1, 2 or 3 different red cliques respectively, so  there is no triangle containing exactly two red edges. Assume there are vertices $u, v, w$, such that $c(uv)=b$ and $c(vw)=y$. Then $u$ and $v$ are in the same red/blue clique and $w$ is in a different red/blue clique, so $c(uw)=y$, and there are no $rby$, $bby$ triangles. \\

\noindent
\textbf{Claim:} $h_2(c) = O(n^{2/3})$.
\begin{proof} It is not difficult to see that $S_{rb}^c = O(n^{2/3})$. Note that any red/yellow clique contains vertices from at most one red clique inside each copy of $K_{n^{2/3}}$. Since each red clique has size $n^{1/3}$ and there are $n^{1/3}$ copies of $K_{n^{2/3}}$, we obtain $S_{ry}^c = O(n^{2/3})$ (see Observation 1). Finally, any blue/yellow clique contains at most one vertex from each red clique inside every copy of $K_{n^{2/3}}$. Since there are $n^{1/3}$ red cliques inside each $K_{n^{2/3}}$, we have $S_{by}^c = O(n^{2/3})$.
\end{proof}

\refstepcounter{constr}\label{const7}
\subsection*{Construction \theconstr\ $(O(n^{2/3}\sqrt{\log n})$, none of $\rby, \rrr, \yyb, \yyr $)} 
Let $k = n^{2/3}\sqrt{\log n}$. From the lower bound on $R(3, t)$/Observation 2 consider a red/blue edge-coloring $c'$ of $K_k$ with no red $K_3$ and no blue $K_t$ for any $t > 9\sqrt{k\log k}$. Let $V_1, \ldots, V_k$ be pairwise disjoint sets each of size $n/k$, and consider the $(V_1, \ldots, V_k)$-blowup of $c'$. Color all edges with both endpoints in $V_i$ yellow for each $i \in [k]$ and call the resulting coloring $c$.\\
Observe that each triangle is either $yyy$ (all vertices in the same $V_i$), $yrr$ or $ybb$ (one edge within a $V_i$ and two edges to a different $V_i$) or one of $rrb, bbr, bbb$ (all vertices in different $V_i$'s).\\

\noindent
\textbf{Claim:} $h_2(c) = O(n^{2/3}\sqrt{\log n})$.
\begin{proof}
Since any red/blue clique can contain at most one vertex from each $V_i$, we have 
$S_{rb}^c = k$. Further, as $c'$ contains no red triangle, we obtain $S_{ry}^c \le 2\cdot(n / k)$. Lastly, as $c'$ contains no blue $K_t$ for any $t > 9\sqrt{k\log k}$, we have that $S_{by} \le 9\sqrt{k\log k}\cdot (n/k)$. Thus, by our choice of $k$ it follows that $\max\{S_{rb}^c, S_{ry}^c, S_{by}^c\} \le  Cn^{2/3} \sqrt{\log n}$, for some constant $C$.
\end{proof}

\refstepcounter{constr}\label{const9}
\subsection*{Construction \theconstr\ ($O(n^{2/3}\sqrt{\log n})$, none of $\rby, \rrr, \rrb, \yyb $)} 
Let $k = n^{1/3}/\sqrt{\log n}$ and consider the trivial edge-coloring $c'$ of $K_k$ where every edge is colored blue. Let $V_1, \ldots, V_k$ be pairwise vertex disjoint sets of size $n/k$ each and consider the $(V_1, \ldots, V_k)$-blowup of $c'$. From the lower bound on $R(3, t)$, there is a red/yellow edge-coloring $c''$ of $K_{n/k}$ with no red $K_3$ and no yellow $K_t$ for any $t > 9\sqrt{(n/k)\log(n/k)}$ and sufficiently large $n$. Color the edges inside $V_i$ according to $c''$ for all $i \in [k]$, and call the resulting coloring $c$.\\
Observe that each triangle is either one of $yyy$, $yyr$ or $rry$ (all vertices in the same $V_i$),
one of $bbr, bby$ (vertices in two different $V_i$'s) or $bbb$ (vertices in three different $V_i$'s).  \\

\noindent 
\textbf{Claim:}  $h_2(c) = O(n^{2/3}\sqrt{\log n}).$
\begin{proof}
First, note that any red/yellow clique can contain vertices from at most one of the $V_i$'s. It easily follows that $S_{ry}^c = n/k$. Furthermore, using the bounds on the sizes of red and yellow cliques inside any given $V_i$, we obtain $ S_{by}^c \le k\cdot 9\sqrt{(n/k)\log(n/k)}$,  and $S_{rb}^c \le 2k$. Thus, our choice of $k$ implies that $\max\{S_{rb}^c, S_{ry}^c, S_{by}^c\} \le C n^{2/3} \sqrt{\log n}$, for some constant $C$.
\end{proof}

\refstepcounter{constr}\label{const:2/3_no_k3}
\subsection*{Construction \theconstr\ ($O(n^{2/3}\sqrt{\log n})$, none of $\rrr, \rrb, \bbr, \bby$)} 

Let $k = n^{2/3}\sqrt{\log n}$ and consider $m = n/k$ pairwise vertex disjoint copies of $K_k$ with vertex sets $V_1, \ldots, V_{m}$ where $V_i = \{v_{i1}, \ldots, v_{ik}\}$ for each $i = 1, \ldots, m$. Let $c'$ be an edge-coloring of $K_k$ in red/yellow with no red $K_3$ and no yellow clique larger than $9\sqrt{k\log k}$. Color the cliques induced by $V_i$ according to $c'$ for each $i = 1, \ldots, m$. For $i \neq j$ color the edge $v_{is}v_{jt}$ blue if $s = t$ and yellow otherwise. Call this final coloring $c$.\\
Observe that there is no $rrr$ and no $rrb$ triangle since $c'$ contains no red $K_3$, and any two incident red edges are completely contained in some $V_i$, which contains no blue edges. Similarly, any two incident blue edges are contained in a blue clique, so there are no $bbr$ and no $bby$ triangles.  \\

\noindent 
\textbf{Claim:}  $h_2(c) = O(n^{2/3}\sqrt{\log n}).$
\begin{proof}
Consider the sets $U_t = \{v_{it}: i = 1, \ldots, m\}$ for $t = 1, \ldots, k$. Observe that any red/yellow clique can contain at most one vertex from each $U_t$. Therefore, we obtain $S_{ry}^c \le k$. From our bound on the sizes of largest yellow cliques in each $V_i$, it follows that 
 $S_{by}^c \le (n/k)\cdot 9\sqrt{k\log{k}}$. Lastly, since there is no clique using both red and blue edges, there is no red triangle and each blue clique has size 
   $n/k$, so we have that $S_{rb}^c = n/k$. Thus, our choice of $k$ implies  $\max\{S_{rb}^c, S_{ry}^c, S_{by}^c\} \le C n^{2/3} \sqrt{\log n}$, for a constant $C$.
\end{proof}

\refstepcounter{constr}\label{const:3/4_no_rb}
\subsection*{Construction \theconstr\ ($O(n^{3/4}\sqrt{\log n})$, none of $\rby, \rrr, \bbb, \rrb$)} 

Let $k = \sqrt{n}$. By the lower bound on $R(3, t)$, take a blue/yellow edge-coloring $c'$ of $K_k$ with no blue $K_3$ and no yellow clique of size greater than $9\sqrt{k\log k}$. Let $V_1, \ldots, V_k$ be pairwise vertex disjoint sets each of size $n/k$ and consider the $(V_1, \ldots, V_k)$-blowup of $c'$. Color each $V_i$ in red/yellow with no red $K_3$ and no yellow clique of size greater than $9\sqrt{(n/k)\log(n/k)}$. Let the resulting coloring be $c$.\\
 Note that $c$ contains neither a monochromatic red triangle nor a monochromatic blue triangles, since $c'$ contains no blue $K_3$ and the coloring inside each $V_i$ contains no red $K_3$. Since the $V_i$'s contain no blue edges, there is no $rrb$ triangle, and it is easy to see that there is no rainbow triangle in $c$.\\

\noindent 
\textbf{Claim:}  $h_2(c) = O(n^{3/4}\sqrt{\log n}).$
\begin{proof}
As there is no blue $K_3$ in $c'$, any red/blue clique contains vertices from at most two of the $V_i$'s, and since there is no red $K_3$ inside the $V_i's$, we see that $S_{rb}^c \leq 4$. Moreover, since yellow cliques in each $V_i$ have size at most $9\sqrt{(n/k)\log(n/k)}$ and the edges between any two distinct $V_i$ and $V_j$ are colored in blue/yellow, we obtain $S_{by}^c \leq k\cdot 9 \sqrt{(n/k) \log(n/k)}$. Finally, since any yellow clique in $c'$ has size at most $9\sqrt{k\log k}$, it follows that $S_{ry}^c \le 9 \sqrt{ k \log k} \cdot (n/k)$. By our choice of $k$ we have $\max\{S_{rb}^c, S_{by}^c, S_{ry}^c\} \le Cn^{3/4}\sqrt{\log n}$, for some constant $C$.
\end{proof}

\refstepcounter{constr}\label{const:2n/5_2}
\subsection*{Construction \theconstr\ ($2\floor{\frac{n}5} + \epsilon$, none of $\rrr,\bbb, \rry, \bby$)} 
Consider the red/blue coloring $c'$ of $K_5$ with no monochromatic triangle. Let $n \ge 5$ be an integer, take $\ceil{\frac n5}$ pairwise vertex disjoint copies of $K_5$ colored according to $c'$,  and delete some vertices from one of these copies to make sure that the total number of vertices is $n$.
Finally, color all remaining edges between these copies yellow. Denote by $c$ the resulting coloring.\\ 
Observe that there are no monochromatic red or blue triangles and that each triangle contains either no yellow edges (if it is contained in a red/blue $K_5$) or at least two yellow edges (if it contains vertices of at least two distinct red/blue $K_5$'s.)   \\

\noindent 
\textbf{Claim:}  $h_2(c) \le 2 \floor{\frac{n}5}  + \epsilon,$ where  $\epsilon = 0$ if $n \equiv 0 \pmod 5$, $\epsilon = 1$ if $n \equiv 1$, and $2$ otherwise.

\begin{proof}
The largest red/blue clique has size $5$, and any red/yellow or blue/yellow clique contains at most two vertices from each of the 
$\floor{n/5}$ copies of $K_5$  and at most $\epsilon$ vertices from the remaining vertices, so we obtain
$S_{by}^c , S_{ry}^c\le 2 \floor{\frac{n}{5}} + \epsilon$.
\end{proof}

\refstepcounter{constr}\label{const:2n/5}
\subsection*{Construction \theconstr\ ($\leq 2\ceil{\frac{n}5}$, none of  $\rrr,\bbb, \yyr, \yyb$)} 
Consider the red/blue coloring $c'$ of $K_5$ with no monochromatic triangle. Let $n \ge 5$ be an integer and let $V_1, \ldots, V_5$ be pairwise disjoint sets of sizes $\ceil{\frac n5}$ or $\floor{\frac n5}$ such that $\sum_{i = 1}^5 |V_i| = n$. Consider the $(V_1, \ldots, V_5)$-blowup of $c'$ and color every edge within $V_i$ yellow for $i = 1, \ldots, 5$. Denote by $c$ the resulting coloring.\\
Observe that there is no monochromatic red or blue triangle in $c$ and that the yellow graph forms a disjoint union of cliques, so there is no triangle with exactly two yellow edges. \\

\noindent 
\textbf{Claim:}  $h_2(c) \le 2\ceil{\frac{n}5}.$
\begin{proof}
The largest red/blue clique has size $5$, and any red/yellow or blue/yellow clique contains at most two of the parts $V_i$'s, so we obtain $S_{by}^c , S_{ry}^c\le 2 \ceil{\frac{n}{5}}$.
\end{proof}

\refstepcounter{constr}\label{const:3n/7}
\subsection*{Construction \theconstr\ ($\le {\ceil{\frac {3n}7}} + 1$, none of $\rrr, \bbb, \bbr, \yyb $)} 

Consider $K_7$ with vertex set $\{v_0, \ldots, v_6\}$. Define a red/blue/yellow edge-coloring $c'$ of $K_7$ as follows. For distinct $i, j \in \{0, \ldots, 6\}$ set\\
\begin{minipage}[t]{0.65\linewidth}
	\vspace{0pt}
$$c'(v_iv_j) = \begin{cases}  b, & \text{ if } i-j=\pm 1 \pmod 7 , \\
y, & \text{ if } i-j =\pm 2\pmod 7, \\
r ,& \text{ if } i-j =\pm 3\pmod 7. \\
\end{cases}\\[0.2cm]$$ 	
\end{minipage}\quad\begin{minipage}[t]{0.2\linewidth}
\vspace{-12pt}
\kseven
\end{minipage}

Note that $c'$ contains no monochromatic blue or red triangles, since the blue and red graph form $7$-cycles. Since vertices at distance $2$ along the cycle are colored yellow, $c'$ contains no $bbr$ triangles. Finally, consider triangles containing two incident yellow edges. By symmetry, we may assume that the vertices of this triangle are $0, 2,$ and $5$. Then since $5-2 = 3$, the third edge must be red. Thus, there are no $yyb$ triangles in $c'$.

Now, let $n \ge 7$ be an integer and let $V_0, \ldots, V_6$ be pairwise disjoint vertex sets $V_0,\ldots, V_6$ of sizes $x=\floor{\frac {n}7}$ or $w=\ceil{\frac {n}7}$ such that $\sum_{i=0}^6 |V_i| = n$. The parts of sizes $x$ and $w$ are arranged cyclically according to the following orders depending on when $n$ is $0, 1, 2, 3, 4, 5, 6$ modulo $7$, respectively: $xxxxxxx$, $wxxxxxx$, $wxwxxxx$, $wwxxwxx$, $wwxxwwx$,  $wwwwwxx$, $wwwwwwx$. Consider the $(V_0, \ldots, V_6)$-blowup of $c'$ and color every edge within $V_i$ yellow for $i = 0, \ldots, 6$. Call the resulting coloring $c$, and note that $c$ still doesn't contain any of the stated patterns.\\

\textbf{Claim:} $h_2(c) \le \ceil{\frac{3n}7} + \epsilon_1(n)$, where $\epsilon_1(n) = 1$ if $n \equiv 2 \pmod{7}$ and $\epsilon_1(n) = 0$, otherwise.
\begin{proof}
If some clique contains vertices from four different $V_i$'s, then it induces all three colors. 
Thus,  any  two-colored clique contains vertices from at most three different $V_i$'s. To prove the upper bound we may assume $n$ is not divisible by $7$. Let us call a partition set $V_i$ \emph{big} if it has size $w$, and otherwise \emph{small}. Write $n = 7x + r$ for non-negative integers $x, r$ with $1 \le r \le 6$. Given our distribution of sizes, it is not difficult to check that any $2$-colored clique contains vertices from at most one big set if $r=1$, at most two big sets if $2 \le r \le 4$, and at most three big sets if $r = 5, 6$. Thus, if $r = 1$ the largest $2$-colored clique has size $2x + w = 3x + 1 = \ceil{3n/7}$. If $r = 3, 4$, then the largest $2$-colored clique has size $2w + x = 3x + 2 = \ceil{3n/7}$. If $r = 5, 6$, the largest $2$-colored clique has size $3w = 3x + 3 = \ceil{3n/7}$. Lastly, suppose $r = 2$. The largest $2$-colored clique has size $2w + x = 3x + 2$. On the other hand, $\ceil{3n/7} = \ceil{3x + 6/7} = 3x + 1$. Hence, the largest $2$-colored clique has size $\ceil{3n/7} + 1$. This completes the proof of the claim. 
\end{proof}

\refstepcounter{constr}\label{const2}
\subsection*{Construction \theconstr\  ($\ceil{\frac n2}$, none of $ \rrr, \yyy, \rrb, \rry, \bbr, \bby, \yyr$)}

Consider the following 3-edge-coloring $c$ of $K_n$: 
take disjoint blue cliques $V_1$, $V_2$ of sizes $\floor{\frac n2}$ and $\ceil{\frac n2}$, put a maximum red matching in between and color all other edges yellow.\\
Observe that each triangle is either monochromatic blue (if it is contained in $V_1$ or $V_2$) or one of $rby$, $yyb$ (if it has w.l.o.g. one vertex in $V_1$ and two in $V_2$).\\

\noindent
\textbf{Claim:} $h_2(c) \le \ceil{\frac n2}$.
\begin{proof} The vertex set of any  blue/red clique of size at least $3$  is contained in either $V_1$ or $V_2$, thus $S_{rb}^c \le \ceil{\frac n2}$. Any red/yellow clique contains at most one vertex from each of $V_i$'s, $i=1,2$, so we have $S_{ry}^c = 2$. Consider a largest blue/yellow clique $X$. Then it has $x_1$ vertices in $V_1$, each one has a red neighbor in $V_2$, so $X$ can contain at most $|V_2| - x_1$ vertices in $V_2$, i.e., we have $S_{by}^c \le |V_2| = \ceil{\frac n2}$. \end{proof}

\refstepcounter{constr}\label{const3}
\subsection*{Construction \theconstr\  ($\ceil{\frac n2} +1, $ none of $\yyy, \rby, \rrb, \rry, \bbr,\bby $)}

Consider the following $3$-edge-coloring $c$ of $K_n$ for $n \ge 3$: 
Take a  red clique $V_1$ of size $\ceil{\frac n2}$ and vertex-disjoint from it a  blue clique $V_2$ of size $\floor{\frac n2}$  and color all edges between $V_1$ and $V_2$ yellow. \\
Observe that each triangle is either monochromatic red or blue (if it is contained in $V_1$ or $V_2$) or one of $yyr, yyb$ (if it has vertices in both $V_1$ and $V_2$).\\

\noindent
\textbf{Claim:} $h_2(c) \le \ceil{\frac n2} +1$.
\begin{proof}
Any red/blue clique contains only vertices from either $V_1$ or $V_2$.
Any red/yellow clique contains at most one vertex from $V_2$ and any blue/yellow clique contains at most one vertex from $V_1$, so we have $\max\{S_{rb}^c, S_{ry}^c, S_{by}^c\} \le |V_1| +1$. \end{proof}

\section{Forbidding $\h$ with $|\h| = 1$} \refstepcounter{sec}

We start with forbidding only one pattern, i.e., up to swapping colors we only need to consider 3 families $\mathcal H \in \{ \{\rby\}, \{\rrr\}, \{\rrb \} \}$.
According to Fox et al. \cite{FGP}, if $\mathcal H = \{\rby\}$, then $h_2(n,\mathcal H) = \Theta(n^{1/3}\log^2 n)$.

\begin{lemma}\label{lem1.rrr}
Let $\mathcal H = \{\rrr\}$. Then we have 
$h_2(n, \mathcal H) = \Theta(\sqrt{n\log n}) .$
\end{lemma}
\begin{proof}
Consider an arbitrary  red/blue/yellow edge-coloring $c$ of $K_n$ that has no red triangle. 
By the upper bound on the Ramsey number $R(3,k)$ (\Cref{thm:AKS})  we see that $c$ contains a blue/yellow  $K_k$ with $k =\Omega (\sqrt{n\log n})$, so $h_{2}(n, \{rrr\}) =\Omega (\sqrt{ n\log n}).$ Moreover, Construction~\ref{const4} shows that $h_2(n, \{rrr\}) = O(\sqrt{n\log n})$. 
\end{proof}

\begin{lemma}\label{lem1.rrb}
Let $\mathcal H =  \{\rrb\}$. Then we have $h_2(n,\mathcal H) = \ceil{\sqrt n}$.
\end{lemma}
\begin{proof}
For the lower bound,  let $c$ be an $\mathcal H$-avoiding red/blue/yellow edge-coloring of $K_n$. Let $\Delta_r$ be the maximum red degree in $c$. If $\Delta_r \ge \floor{\sqrt n}$, there exists a vertex $v$ which has a red neighborhood $N_r(v)$ of size at least $\floor{\sqrt n}$. Note that this neighborhood does not contain a blue edge, since together with $v$ that would create an $rrb$ triangle. Thus, $N_r(v) \cup \{v\}$ spans a red/yellow clique of size at least $\floor{\sqrt n} + 1 \ge \ceil{\sqrt{n}}$. Otherwise, we have $\Delta_r \le  \floor{\sqrt n} - 1$ and so the graph induced on red edges may be vertex-colored with $\floor{\sqrt n}$ colors. One color class has size at least $\ceil{ n / \floor{\sqrt{n}}} \ge \ceil{\sqrt{n}}$. This forms a blue/yellow clique of size at least $\ceil{\sqrt{n}}$ as required. Finally, the upper bound on $h_2(n, \mathcal{H})$ follows from Construction \ref{const1}.
\end{proof}

\section{Forbidding $\mathcal H$ with $|\mathcal H| = 2$} \refstepcounter{sec}

\setcounter{proposition}{1} 
\begin{proposition}\label{prop:2}
 Any family $\mathcal H$ consisting of two distinct patterns can be obtained by applying a color permutation to all patterns in one of the following families: 
 \begin{itemize}
 	\item $\{\rrb, \rry\}, \{\rrb, \bbr\},  \{\rrb, \bby\}, \{\rry, \bby\}$,
 	\item $\{\rrb, \rby\}$,
 	\item $\{ \rrr, \bby\}, \{ \rrr, \bbr\} , \{\rrr,\rrb\}, \{\rrr, \bbb\}$, 
 	\item $\{\rrr, \rby\}.$
 \end{itemize}
\end{proposition}
We prove this proposition in the Appendix.

\subsection{$\mathcal H$ contains no rainbow and no monochromatic pattern}

\begin{lemma}\label{lem2.1}
 Let $\h \in \{ \{\rrb, \rry\}, \{\rrb, \bbr\},  \{\rrb, \bby\}, \{\rry, \bby\}\}$. Then we have $h_2(n,\mathcal H) = \ceil{\sqrt n}$.
\end{lemma}
\begin{proof}
For the lower bound note that we have  $\{rrb\} \subseteq \mathcal H$ or $\{rry\} \subseteq \mathcal H$ for all of these families. Thus, by $h_2(n, \{rrb\}) = h_2(n, \{rry\}) $  and \Cref{lem1.rrb}, we obtain  $h_2(n, \mathcal H) \ge h_2(n, \{rrb\}) = \ceil{\sqrt n}$. The upper bound follows from Construction \ref{const1}.  \end{proof}

\subsection{$\mathcal H$ contains a rainbow but no monochromatic pattern}

\begin{lemma}\label{lem2.2}
Let $\mathcal H = \{\rrb, \rby\}$. Then we have $h_2(n, \mathcal H) = \Theta(\sqrt{n}).$
\end{lemma}

\begin{proof}
The lower bound follows from \Cref{lem1.rrb}: $h_2(n, \h) \ge h_2(n, \{rrb\}) \ge \ceil{\sqrt n}$.	

The upper bound follows from  Construction \ref{const:sqrt_n_k4}.
\end{proof}

\subsection{$\mathcal H$ contains a monochromatic but no rainbow pattern}

Since our forbidden family contains a monochromatic triangle, by \Cref{lem1.rrr} we have the lower bound  $h_2(n, \mathcal H) \ge h_2(n, \{rrr\}) = \Omega(\sqrt{n\log n})$. 

\begin{lemma}\label{lem2.3}
Let $\mathcal H \in \{\{ \rrr, \bby\}, \{ \rrr, \bbr\} \}$. Then we have $h_2(n, \mathcal H)= \Theta(\sqrt{n\log n})$.
\end{lemma}
\begin{proof}
Since $\{rrr\} \subseteq \mathcal H$,   \Cref{lem1.rrr} implies that $h_2(n, \mathcal H) \ge h_2(n, \{rrr\}) = \Omega(\sqrt{n\log n}) $.
The upper bound follows from Construction \ref{const4}.
\end{proof}

Recall, that for a graph $G$,  $f(G)= \max \{\alpha(G),\omega(G^2)\}$ and $f(n)= \min \{f(G): ~ |G|=n,~ \omega(G)=2\}$. The following lemma shows that determining the value of $f(n)$ is closely linked to determining the value of $h_2(n, \mathcal{H})$, where $\mathcal{H} = \{\rrr, \rrb\}$.

\begin{lemma} \label{lem2.f1}
Let $\h = \{\rrr, \rrb\}$. Then $f(n) \leq h_2(n, \h) \leq 2  f(n)$.
\end{lemma}
\begin{proof}
For the upper bound consider a triangle-free graph $G$ on $n$ vertices such that $f(G)=f(n)$, i.e., $\alpha(G)\leq f(n)$ and 
$\omega(G^2) \leq f(n)$.  Color the edges of $G$ red, color each edge from $E(G^2)\setminus E(G)$ yellow, and color all remaining edges blue. 
We see that there are no red triangles and any two adjacent red edges form an $rry$ triangle. 
Note that $S_{by} = \alpha(G) \leq f(n)$,  $S_{ry}^c= \omega (G^2) \leq f(n)$ and $S_{rb}^c \leq 2\alpha (G)\leq 2f(n)$. 
Here, the statement on $S_{rb}^c$ holds since in any red/blue clique, the red graph forms a matching.

For the lower bound, consider an arbitrary $\h$-avoiding coloring $c$ of $K_n$. Let $G$ be the red graph. Then $S_{by} = \alpha(G)$.
Since there is no $rrb$ triangle, each triangle containing two red edges is an $rry$ triangle, so $S_{ry}^c \geq  \omega(G^2)$.
Thus $h_2(c) \geq \max\{ \alpha(G),  \omega(G^2)\} \geq f(n)$.   
\end{proof}

\begin{lemma}\label{lem2.4}
Let $\mathcal H =\{\rrr, \rrb\}$. Then we have $$h_2(n, \h) = \Omega(\sqrt{n\log n}) \text{ and } 
h_2(n, \h) = O(\sqrt{n} \log^{3/2} n).$$
\end{lemma}
\begin{proof} 
The lower bound follows from the lower bound  on $h_2(n, \{rrr\})$,  \Cref{lem1.rrr}.
The upper bound follows from  Construction \ref{const:r-b-only}.
\end{proof}

\begin{lemma}\label{lem2.5}
Let $\mathcal H =\{\rrr, \bbb\}$. Then we have $h_2(n, \h) = \Theta(\sqrt{n\log n} )$. 
\end{lemma}
\begin{proof}
The lower bound follows from the lower bound  on  $h_2(n, \{rrr\})$, \Cref{lem1.rrr}.\\

The upper bound follows from a result by Guo and Warnke \cite{GW}, that implies that there are two edge-disjoint triangle free subgraphs $G$ and $G’$ of $K_n$, each with independence number at most $c \sqrt{n\log n}$. We color the edges of $G$, the edges of $G’$ blue and the rest of the edges yellow. 
Then $S_{yr}^c = \alpha(G’)\leq c \sqrt{n\log n}$, $S_{yb}^c=\alpha(G)\leq c \sqrt{n\log n}$, and $S_{rb}^c \leq 6$ since red and blue graphs are triangle-free.\\

Here, we include a proof of a slightly weaker bound ($h_2(n, \h) = O(\sqrt n \log n)$), which is easily obtained by random packing:\\
We shall find two edge-disjoint triangle-free graphs each with sufficiently small independence number. We shall color the edges of one of them red, the other one blue, and the rest of the edges yellow.  Then the largest bicolored clique will have size at most the size of the largest independence set of each of these triangle-free graphs.

	As we can not directly guarantee that the desired packing exists, we shall deal with a small overlap. Consider a graph $G$ on a vertex set $[N]$  that is triangle-free, such that $\alpha(G)= O(\sqrt{N \log N})$. In particular, we have that $\Delta(G) =O(\sqrt{N \log N})$ and $|E(G)| =O(N^{3/2} \sqrt{\log N})$. The following claim asserts that we can find a copy $G'$ of $G$ on vertex set $[N]$ such that $|E(G)\cap E(G')|$ is small. A similar proof appears in Konarski and \.Zak~\cite{KZ}, but we include the short proof in the Appendix for convenience of the reader.
\begin{claim} There is a copy $G'$ of $G$ on vertex set $[N]$ such that $|E(G)\cap E(G')| \le |E(G)|^2 / \binom{N}{2}$.
\end{claim}

Consider the union of $G$ with its isomorphic image $G'$ on $[N]$ granted by the above Claim.  Since $|E(G)| = O(N^{3/2} \sqrt{\log N})$, we obtain $|E(G)\cap E(G')| \le |E(G)|^2 / \binom{N}{2} = O(N\log N)$. Let $G''$ be the graph on $[N]$ with edge set $E(G)\cap E(G')$.
Then it has at least $N/2$ vertices with degree  $O(\log N)$. These vertices induce a graph with independence number at least $\Omega(N/\log N)$, so $\alpha(G'') = \Omega(N/\log N)$.  Let $X$ be a largest independent set in $G''$, and let $N$ be selected such that  $n= |X|$. In particular, $n= \Omega (N/\log N)$, and $N= O(n \log n)$.   Now, color the edges of $G[X]$ red, edges of $G'[X]$ blue and the rest yellow.
	We see that $S_{rb}\leq 5$ since there are no  red and no blue triangles. We have that $S_{by} \leq \alpha(G)$ since any blue/yellow  clique corresponds to an independent set in $G$. Thus,
\[
	S_{by} =O(\sqrt{N\log N})= O(\sqrt{n \log n \log (n \log n)})= O( \sqrt{n} \log n),
\]
and the lemma follows.
\end{proof}

\subsection{$\mathcal H$ contains a rainbow and a  monochromatic pattern}

\begin{lemma}\label{lem2.6}
Let $H= \{ \rrr, \rby\}$. Then $h_2(n, \h) = \Omega(n^{2/3}/ \log^{3/2}n)$ and $ h_2(n, \h) = O(n^{2/3} \sqrt{\log n})$.
\end{lemma}
\begin{proof}
The upper bound follows  from Construction \ref{const7}.\\

For the lower bound, consider an arbitrary $\h$-avoiding coloring $c$ of $K_n$.  
	Consider the vertex sets of red components, which we refer to as \emph{blobs}.
	Note that all edges between any two blobs are of the same color, either blue or yellow, otherwise there is a rainbow triangle in $c$.
	Since the coloring induced by each blob is Gallai, by \Cref{thm:G} we have that each blob is a disjoint union of sets which we call {\it sub-blobs}, so that 
	all edges between any two sub-blobs are of  the same color and the total number of colors between sub-blobs is at most $2$. Note that since each blob is a red connected component, one of the colors between sub-blobs must be red and another is blue or yellow. Note also that each sub-blob spans a blue/yellow clique. Otherwise, there is a red triangle in $c$.\\

	We shall delete some vertices of the graph such that $c$ restricted to the remaining part is easier to analyze. 
	Specifically, we will end up with a coloring $c''$ of a complete graph on at least $C''n/\log^2 n$ vertices (for some constant $C''$) in which all blobs contain the same number of sub-blobs and all sub-blobs overall have the same size. In addition, this coloring will have only red and blue edges between sub-blobs of any given blob. \\

	We can assume that each blob has at least two vertices because if there are at least $n/2$  blobs of size $1$, 
	they correspond to a blue/yellow clique on at least $n/2$ vertices. 
	It could be assumed, without loss of generality,  that non-red edges between sub-blobs of any given blob are blue. Indeed, either at least $n/4$ vertices  are spanned by blobs with red/blue between the sub-blobs or at least $n/4$ vertices  are spanned by blobs with red/yellow between sub-blobs. Let us assume the former. \\
	
	We shall split the sub-blobs according to sizes. Let $X_i$ be the union of all sub-blobs of sizes from $2^i$ to $2^{i+1}-1$, $i=0, \ldots,  \log n$.
	Consider $i$ for which $X_i$ is largest, i.e., $|X_i|\geq \frac14 n/ \log n$.
	Delete at most half of the vertices from sub-blobs in $X_i$ so that all of them are of the same size, and call the resulting set $X_i'$.
	Now, consider a coloring $c'$ that is the restriction of $c$ to $X_i'$. We see that $c'$ has the same structure as $c$ but with all sub-blobs of the same size and total number of vertices  $n'\geq \frac18 n/\log n$.  Let $k$ be the size of each sub-blob.
	If $k >n^{2/3}$ we are done since each sub-blob spans a blue/yellow clique, and then $S_{by} > n^{2/3}$.
	Thus, $k<n^{2/3}$.   Let $Y_j$ be the union of blobs each having sizes from $2^j$ to $2^{j+1}-1$, $j=0, \ldots,  \log n'$.
	Consider $Y_j$ of largest size so that $|Y_j| \ge n'/\log n'$.  By deleting at most half of the vertices in $Y_j$ we can assume that all blobs in $Y_j$ have the same number of vertices, and hence exactly the same number of sub-blobs. Denote the number of sub-blobs by $\ell$.
	Again, by restricting $c'$ to the resulting set, we have an $\mathcal{H}$-avoiding coloring $c''$ on $n''\geq C''n/\log^2 n$ vertices (for some constant $C''$) with each blob having $\ell$ sub-blobs
	and each sub-blob having $k$-vertices. Recall that $k<n^{2/3}$.\\

	Now we shall analyze $c''$. 
	Since the red graph is triangle-free, each blob has a blue/yellow clique of order at least $C\sqrt{\ell \log  \ell}\cdot  k$ for some constant $C > 0$, by \Cref{cor:AKS}. Taking a union of these cliques over all blobs, we see that 
	$$S_{by} \geq C\sqrt{\ell \log  \ell}\cdot  k \cdot n''/ k \ell = C\sqrt{ \log \ell}\cdot n''/ \sqrt{\ell} \ge CC'' \sqrt{\log \ell /\ell }\cdot n /(\log^2 n).$$
	Thus, if $\ell <n^{2/3}$, we are done as in this case $S_{by} \geq C' n^{2/3}/ \log^{3/2} n$, for some constant $C'$.
	Thus, $\ell \geq n^{2/3}$. However, in this case pick a blob and pick a vertex from each sub-blob of this blob. This gives a red/blue clique on $\ell \geq n^{2/3}$ vertices. 
\end{proof}

\section{Forbidding $\mathcal H$ with $|\mathcal H| =3$}\refstepcounter{sec}

\begin{proposition}\label{prop:3}
Any family $\mathcal H$ consisiting of 3 distinct patterns can be obtained by applying a color permutation to all patterns in one of the following families: 
\begin{itemize}
	\item $\{\rrb, \rry, \bbr\}$, $\{\rrb, \rry, \bby\}$, $\{\rrb, \bbr, \yyr\}$, $\{\rrb, \bby, \yyr \}$,
	\item $\{\rrb, \rry, \rby \}$,  $\{\rrb, \bbr, \rby \}$,  $\{\rry, \bby, \rby \}$,  $\{\rrb, \bby, \rby \}$,
	\item $\{\rrr,\bbb,\yyy\} , \{\rrr,\bbb,\rrb\}, \{\rrr,\bbb,\rry\}, \{\rrr,\bbb, \yyr\}$, 
	\item $\{\rrr,\rrb, \rry\}, \{\rrr,\rrb, \bbr\}, \{\rrr,\rrb, \bby\},  \{\rrr, \rrb, \yyr\}, \\\{\rrr, \rrb, \yyb\},  
	\{\rrr, \bbr, \bby\}, \{\rrr,\bbr, \yyr\}, \{\rrr,\bbr,\yyb\}, \{\rrr,\bby, \yyb\}$,
	\item $\{\rby, \rrr, \bbb\} $, $\{\rby, \rrr, \rrb\}  $, $\{\rby, \rrr, \bbr\}  $, $\{\rby, \rrr, \bby\}  $.
\end{itemize} 
\end{proposition}
The proof of this Proposition is given in Appendix.

\subsection{$\mathcal H$ contains no rainbow and no monochromatic pattern}

\begin{lemma}\label{lem3.1}
Let $\mathcal H \in \{ \{\rrb, \rry, \bbr\}, \{\rrb, \rry, \bby\}\}$. Then we have $h_2(n, \mathcal H) = \ceil{\sqrt n}$.
\end{lemma}

\begin{proof}
The lower bound follows from  \Cref{lem1.rrb} since  $h_2(n, \mathcal H) \ge h_2(n, \{rrb\}) = \ceil{\sqrt n}.$ 
The upper bound follows from Construction \ref{const1}.
\end{proof}

\begin{lemma}\label{lem3.2}
Let $\mathcal H = \{\rrb, \bbr, \yyr\}$. Then we have $h_2(n, \mathcal H) = \ceil{\frac n2}$.
\end{lemma}

\begin{proof}
The upper bound follows from Construction \ref{const2}.\\

\noindent For the lower bound, let $c$ be an $\mathcal H$-avoiding edge-coloring of $K_n$. 

\noindent
Case 1: There is a red triangle in $c$.\\
Let the vertex set of a red triangle be $\{u, v, w\}$. Then there cannot be a blue edge adjacent to the triangle. Assume the contrary, i.e., there is a blue edge $au$.  Then $av$ cannot be red  or blue, since then $uva$ would induce an $rrb$ or a $bbr$ triangle respectively, i.e., $av$ has to be yellow. The same holds for $aw$, but then $vwa$ forms a $yyr$ triangle, a contradiction.  Assume there is a blue edge $xz$ in the graph. 
Since each of $x$ and $z$ send only red and yellow edges to $\{u,v,w\}$, and each of $x$ and $z$ send at most one yellow edge to $\{u,v,w\}$, $x$ and $z$ have a common red neighbor in $\{u,v,w\}$, say $u$. 
 But then $uxz$ is a $rrb$ triangle, a contradiction.  Thus, if the graph contains a red triangle, it contains no blue edge and hence, we have a red/yellow clique of size $n$. \\

\noindent
Case 2: $c$  contains no red triangle.\\
We show that in this case the red graph is bipartite. We need to show that there is no red odd cycle. Assume the contrary, and let $v_1v_2\cdots v_kv_1$ ($k\ge 5$) be a shortest red odd cycle. Then we cannot have any red chord of the cycle, since that would create a shorter red odd cycle. 
Assume there is an index $i$ such that $v_1v_i$ and $v_1v_{i+1}$ have the same color. But then $v_1v_iv_{i+1}$ create a $bbr$ or  a $yyr$ triangle. Also, the edge $v_1v_3$  has to be yellow, since otherwise $v_1v_2v_3$  creates a $bbr$ triangle. Similarly,  $v_1v_{k-1}$ is yellow. But then, combining these two facts we obtain that all edges of the form $v_1v_i$ with $i$ odd are yellow, including $v_1v_{k-2}$. Then $v_{k-2}v_{k-1}v_1$ forms a $yyr$ triangle, a contradiction. 

Thus, we have no odd red cycle, so the red graph is bipartite. But then in any bipartition there is a bipartite class of size at least $\ceil{\frac n2}$ in which only colors blue and yellow appear. Hence, we have a 2-colored set of size $\ceil{\frac n2}$. 
\end{proof}

\begin{lemma}\label{lem3.3}
Let $\mathcal H = \{\rrb, \bby, \yyr \} $. Then we have $h_2(n, \mathcal H) = \ceil{\frac n2}$ for $n \neq 7$ and $h_2(7, \h) = 3$.
\end{lemma}

\begin{proof}
The upper bound follows  from Construction \ref{const2}.\\

For the lower bound, let $c$ be an $\mathcal H$-avoiding coloring of $K_n$ on a vertex set $V$.
Assume first that there is a monochromatic triangle. Because of symmetry on forbidden patterns, we may assume that there is a red triangle.
Let $R$ be a largest red clique in $c$. Then $|R| \ge 3$. Note that every vertex not in $R$ sends a yellow or a blue edge to $R$.
Then every vertex outside of $R$ sends at most one yellow edge to $R$ (otherwise we have a $yyr$ triangle). In addition, no vertex outside of $R$ sends both red and blue edges to $R$, 
otherwise we get an $rrb$ triangle.

Thus, $V-R= O\cup P$, where  $O$ is the set of vertices in $V-R$ such that each  edge between $O$ and $R$ is yellow or red and $P$ is  the set of vertices in $V-R$ such that each edge between $P$ and $R$ is yellow or blue. Note that $O$ and $P$ are disjoint. 

Every vertex from $O$ sends $|R|-1$ red  edges to $R$. Then any two vertices in $O$ have a common red neighbor in $R$, and hence there cannot be a blue edge induced by  $O$. 
Similarly, any two vertices in $P$ have a common blue neighbor in $R$ and hence, there cannot be a yellow edge induced by $P$. 

Thus, we have either $|P | \ge \ceil{\frac n2}$ which yields a red/blue clique of size $\ceil{\frac n2}$ or $|R \cup  O| \ge \ceil{\frac n2}$, which is a red/yellow clique of desired size. 

It remains to deal with the case where we have no monochromatic triangle. In this case, the red neighborhood of any vertex induces a yellow clique, the blue neighborhood induces a red clique,  and the yellow neighborhood induces a blue clique, so the maximum degree at each vertex must be at most  $6$, i.e., we only need to consider colorings of $K_n$ with $n \le 7$. 

For $n=7$ every vertex must have degree $2$ in any color, so each color class is a $2$-factor.  Since there is no monochromatic triangle, each color class must be a $C_7$ and up to isomorphism there is a unique such coloring (see also Construction \ref{const:3n/7}). One can create such a coloring by ordering the vertices cyclically and coloring an edge with vertices at distance $1$, $2$, $3$ along the cycle yellow, red, and blue respectively. In this coloring the largest $2$-colored clique has size $3$. 

For $n \le 6$, observe that there is no red $C_5$,  otherwise  all other edges induced by the vertex set of this $C_5$ are yellow since there are no $rrr$ and no $rrb$ triangles. 
But then there is a $yyr$ triangle. Thus, the red graph has no odd cycles and so is bipartite. Therefore, the blue/yellow graph contains a clique of size $\ceil{\frac n2}$.  
\end{proof}

\subsection{$\mathcal H$ contains a rainbow but no monochromatic pattern}


\begin{lemma}\label{lem3.4}
Let $\mathcal H = \{\rrb, \rry, \rby\}$. Then we have $h_2(n,\mathcal H) = \Omega(\sqrt{n})$ and $h_2(n, \mathcal H) = O(\sqrt{n\log n}).$
\end{lemma}
\begin{proof}
By \Cref{lem2.2} we obtain 
$h_2(n, \mathcal H) \ge h_2(n, \{rrb, rby\}) = \Omega(\sqrt{n})$.
The upper bound follows from  Construction \ref{const5}.
\end{proof}

\begin{lemma}\label{lem3.5}
Let $\mathcal H = \{\rrb, \bbr, \rby\}$. Then we have $h_2(n, \mathcal H) = \ceil{\frac n2}+1$.
\end{lemma}
\begin{proof}
For the lower bound,  let $c$ be an $\mathcal H$-avoiding coloring of $K_n$. 
Assume there is a vertex $v$ incident to a red edge $vx$ and a blue edge $vy$. But then the edge $xy$ cannot be colored. Thus, w.l.o.g at least $\ceil{\frac n2}$ vertices are not incident to a red edge. Taking a maximum set of vertices not incident to a red edge and an arbitrary additional vertex (if it exists) creates a blue/yellow clique, i.e., we have $h_2(n, \mathcal H) \ge \ceil{\frac n2}+1$. \\
The upper bound follows from Construction \ref{const3}.
\end{proof}

\begin{lemma}\label{lem3.6}
Let $\mathcal H = \{\rry, \bby, \rby\}$.  Then we have $h_2(n,\mathcal H) = \Omega(\sqrt{n})$ and $h_2(n, \mathcal H) = O(\sqrt{n\log n}).$
\end{lemma}
\begin{proof}
By \Cref{lem2.2} we obtain 
$h_2(n, \mathcal H) \ge h_2(n, \{rrb, rby\}) = \Omega(\sqrt{n})$.
The upper bound follows from Construction \ref{const6}.
\end{proof}

\begin{lemma}\label{lem3.7}
Let $\mathcal H = \{\rrb, \bby, \rby\}$. Then we have $h_2(n,\mathcal H) = \Theta(n^{2/3})$.
\end{lemma}

\begin{proof}

For the lower bound consider an $\mathcal{H}$-avoiding edge-coloring $c$ of $K_n$ on vertex set $V$. Consider a partition of $V$ into sets $A_1, \ldots, A_m$ such that $A_1$ is maximum sized red clique in $c$, and for each $i \ge 2$, $A_i$ is a maximum sized red clique in $c$ contained in $V - (A_1 \cup \cdots \cup A_{i - 1})$. Note that $|A_i| = 1$ is allowed here. Moreover, for each $i \neq j$ there is at least one non-red edge between $A_i$ and $A_j$. We shall show that either there is a $2$-colored clique of a desired size or there are at least $n/2$ vertices such that the coloring restricted to these vertices is formed by pairwise vertex-disjoint red cliques such that between any two such cliques all edges are blue or all edges are yellow.\\

First, assume there is a blue edge $uv$ with $ u \in A_i$ and $ v \in A_j$, for some $i \neq j$. Then every edge between $A_i$ and $A_j$ incident to $u$ or $v$ must be blue, since otherwise a rainbow or $rrb$ triangle is formed. Assume there is an edge $wz$ with $w \in A_i$ and $z \in A_j$ that is not incident to $uv$. If $wz$ is red, then $uz$ cannot be colored without forming a forbidden pattern. Similarly, $wz$ cannot be yellow. It follows that if there is a blue edge between any $A_i$ and $A_j$, then all edges between $A_i$ and $A_j$ must be blue. \\

We claim that for each $A_i$, either $A_i$ sends red/yellow edges to every other $A_j$, or $A_i$ sends only blue/yellow edges to every other $A_j$. Suppose otherwise, so that there is $i, j, k \in [m]$ such that all edges between $A_i$ and $A_j$ are blue, and all edges between $A_i$ and $A_k$ are red/yellow and there is at least one red edge. We assume first that $k < i$. In this case, $A_k$ was chosen as a largest clique before $A_i$. There must be at least one red and at least one yellow edge between $A_k$ and $A_i$, so there exist vertices $u, v, w$ such that $v \in A_i$, $u, w \in A_k$, and $c(uv) = r$ and $c(vw) = y$. Pick a vertex $z \in A_j$. Then $zv$ is blue, so we cannot use blue between $A_j$ and $A_k$: otherwise $vzw$ is a $bby$ triangle. Similarly, we cannot color $zw$ red, because then $vzw$ is a rainbow triangle. It follows that $zw$ is yellow. But then we cannot color $zu$ without forming an $rrb$ or rainbow triangle. The argument is similar if $i < k$. Indeed, in this case we find vertices $u, v, w$ such that $v \in A_k$, $u, w \in A_i$ and $c(vw) = r$ and $c(vu) = y$. Pick a vertex $z \in A_j$ and note that both $uz$ and $uw$ are blue. In this case, note that if $vz$ is red, then $vzw$ is an $rrb$ triangle. If it's blue, then $uvz$ is a $bby$ triangle, and if it is yellow, then $vzw$ is a rainbow triangle. This is a contradiction. Therefore, we say that $A_i$ is of \emph{Type I} if it sends only blue/yellow edges to all other $A_j$'s. Otherwise, we say that $A_i$ is of \emph{Type II}.

Given the above, we can now break the proof into two cases:\\

\noindent 
Case 1: There are  at least $n/2$ vertices in cliques of Type II. In this case, we have a red/yellow clique of size at least $n/2$. \\

Case 2: At least $n/2$ vertices are in cliques of Type I.\\
Relabel and denote by $V_1,\ldots, V_k$ the red cliques of Type I. Recall that all edges between $V_i$ and $V_j$, $i\neq j$  must be blue or all of them must be yellow. Then we denote by $c(V_i, V_j)$ the color of the edges between $V_i$ and $V_j$.

Note that $k<n^{2/3}$, otherwise we have a blue/yellow clique of that size. \\
Let $\mathcal I \subseteq [k]$ be the subset of indices such that $|V_i| \ge \frac14n^{1/3}$ iff $i \in \mathcal I$.
Split each $V_i$, $i \in \mathcal I$ into disjoint sets $ B_{i,j}$ and  $C_i$, $j=1,\ldots,m_i$, with $|B_{i,j}| = \frac14n^{1/3}$ and $|C_{i}| < \frac 14 n^{1/3}$. Then we have 
$$ \sum\limits_{i\in \mathcal I}|C_i| + \sum\limits_{i \not\in \mathcal I} |V_i| < \frac 14 n^{1/3} n^{2/3} = \frac n4. $$
Thus, there are $s > \frac{n/2-n/4}{ n^{1/3}/4} = n^{2/3}$ sets $B_{i,j}$. Consider a blue/yellow edge-coloring $c'$ of $K_s$ with vertex set $\{B_{i,j}: i\in \mathcal I, j\in [m_i]\}$ where $$c'(B_{i, j}, B_{i', j'}) = \begin{cases}
c(V_{i}, V_{i'}), & i \neq i' \\
blue, & i = i'.
\end{cases} $$ 
Then in $c'$, there is no $bby$ triangle, so in $c'$ there is a monochromatic set of size at least $\sqrt{s}$.
Indeed,  the blue graph in $c'$ is a pairwise vertex-disjoint union of cliques, so either one of these cliques has at least $\sqrt{s}$ vertices, or there are 
at least $\sqrt{s}$ such cliques and thus there is a yellow clique on $\sqrt{s}$ vertices. 
Such a  blue clique in $c'$ corresponds to a blue/red clique in $c$,  such a yellow clique in $c'$ corresponds to a red/yellow clique in $c$  of size $\sqrt s |B_{i,j}| > \sqrt{n^{2/3}}n^{1/3}/4 = n^{2/3}/4$ in $c$. \\

The upper bound follows from Construction \ref{const8}. \end{proof}

\subsection{$\h$ contains a monochromatic but no rainbow pattern}\label{size3mono-norb}

\begin{lemma}
Let $\h = \{\rrr, \bbb, \yyy\}$. Then there is no $\h$-avoiding coloring for $n\ge 17$.
\end{lemma}
\begin{proof}
We know that the Ramsey number $R(3,3,3)=17$, so there is no edge-3-coloring of $K_n$ without a monochoromatic $K_3$ for $n\ge 17$. 
\end{proof}

\begin{lemma}\label{lem3.rrr.bbb.rrb}
 Let $\mathcal{H} = \{\rrr, \bbb, \rrb\}$. Then $h_2(n, \mathcal{H}) = O(\sqrt{n} \log^{3/2} n)$ and $h_2(n, \mathcal{H}) = \Omega(\sqrt{n\log n})$.
\end{lemma}
\begin{proof}
 The lower bound holds since $h_2(n, \mathcal{H}) \ge h_2(n, \{rrr\}) = \Omega(\sqrt{n \log n}).$
The upper bound follows from Construction \ref{const:r-b-only}.
\end{proof}

\begin{lemma}\label{lem3.rrr.bbb.rry}
Let $\mathcal{H} = \{ \rrr, \bbb, \rry\}$. Then $h_2(n, \mathcal{H}) = 2\floor{\frac{n}5} + \epsilon$, where $\epsilon = 0$ if $n \equiv 0 \pmod 5$, $\epsilon=1$ if $n \equiv 1 \pmod 5$, and $\epsilon= 2$ otherwise.
\end{lemma}
\begin{proof}
To see the lower bound,  consider an $\mathcal{H}$-avoiding coloring of $E(K_n)$. Observe that the red degree of every vertex is at most $2$. Indeed, since there are no $rry$ or $rrr$ triangles, the entire red neighborhood of a given vertex must induce a blue clique. But as there is no blue triangle, each red neighborhood has at most $2$ vertices.  Thus each red component is either a path or a cycle of length at least $4$. 
Among all such red graphs, the one with the smallest independent set is a union of pairwise vertex disjoint $C_5$'s, and if $n$ is not divisible by $5$, a component on at most $4$ vertices. This matches exactly the Construction \ref{const:2n/5_2} and 
gives a blue/yellow clique of size at least $2\floor{n/5} + \epsilon$.\\

The upper bound follows from Construction \ref{const:2n/5_2}.
\end{proof}

\begin{lemma}\label{lem2}
If a family $\h$ contains three patterns with different majority colors, then $h_2(n, \h) \ge \ceil{\frac{n-1}3}$.
\end{lemma}
\begin{proof}
Consider an	$\h$-avoiding coloring $c$ and an arbitrary vertex $v$. Let $N_r, N_b$, and $N_y$ be the red, blue, and yellow neighborhoods of $v$, respectively. Then we see that each of these sets induces a $2$-colored clique.  Since at least one of the sets $N_r, N_b, N_y$ has size at least $\ceil{\frac{n-1}3}$, the result follows. 
\end{proof}

\begin{lemma}\label{lem3.rrr.bbb.yyr}
Let $\mathcal{H} = \{\rrr, \bbb, \yyr\}$. Then $\ceil{\frac {n-1}3} \le h_2(n, \mathcal{H}) \le 2\ceil{\frac{n}5}$.
\end{lemma}
\begin{proof}
The lower bound follows from \Cref{lem2}. 

The upper bound follows from Construction \ref{const:2n/5}.
\end{proof}

\begin{lemma}\label{lem3.rrr.rrb.rry}
Let $\h = \{\rrr,\rrb,\rry \}$. Then we have $h_2(n, \h) = \ceil{\frac n2}$.
\end{lemma}
\begin{proof}
For the lower bound, let $c$ be an $\h$-avoiding coloring of $K_n$. Then there are no two adjacent red edges, so the red graph forms a matching, i.e., there is a blue/yellow clique of size $\ceil{\frac n2}$.   

The upper bound follows from Construction \ref{const2}.
\end{proof}

\begin{lemma}\label{lem3.rrr.rrb.bbr}
Let $\h = \{\rrr,\rrb,\bbr\}$. Then  $h_2(n, \h) = \Omega(\sqrt{n \log n})$ and 
$h_2(n, \h) = O(\sqrt{n} \log^{3/2} n)$. 
\end{lemma}
\begin{proof}
The lower bound follows from \Cref{lem1.rrr}: $h_2(n, \h) \ge h_2(n, \{rrr\}) = \Omega(\sqrt{n \log n})$.
The upper bound follows from Construction \ref{const:r-b-only}.
\end{proof}

\begin{lemma}\label{lem3.rrr.rrb.bby}
Let $\h = \{\rrr,\rrb,\bby\}$. Then $h_2(n, \h) = \Omega(\sqrt{n \log n})$ and   $h_2(n, \h) = O(n^{2/3}\sqrt{\log n}  )$. 
\end{lemma}
\begin{proof}
The lower bound follows from \Cref{lem1.rrr}: $h_2(n, \h) \ge h_2(n, \{rrr\}) = \Omega(\sqrt{n \log n})$. The upper bound follows from Construction \ref{const:2/3_no_k3}.
\end{proof}

Although there is a gap between the lower and the upper bound in \Cref{lem3.rrr.rrb.bby}, we are able to prove the following lemma concerning the structure of colorings with no patterns in $\{\rrr,\rrb,\bby\}$.

\begin{lemma}\label{lem:rrr.rrb.bby-structure}
Let $\h = \{\rrr,\rrb,\bby\}$ and let $c$ be an $\h$-avoiding coloring of $K_n$. Then either $h_2(c) = \Omega(n^{2/3}\log^{1/3} n)$ or at least $n/4$ vertices span pairwise vertex-disjoint blue cliques with only red/yellow edges between them, with the red graph forming a matching  between any two distinct blue cliques. 
\end{lemma}
\begin{proof}
See appendix.
\end{proof}

Recall that $g(n)$ is a smallest possible independence number of an $n$-vertex graph that has no cycles of length $3$ and no cycles of length $5$, i.e., that has an odd girth at least $7$. 

\begin{lemma}\label{lem:g(n)}
We have that  $g(n)\leq  h_2(n, \{\rrr, \rrb, \yyr\})\leq 2 g(n)$.
\end{lemma}

\begin{proof}
For the upper bound, let $G$ be an $n$-vertex graph with odd girth at least $7$ and independence number $g(n)$. Color its edges red, the edges corresponding to pairs of vertices at distance two in $G$ yellow, and all remaining edges blue. Clearly, we have no $rrr$ and no $rrb$ triangles. Assume that there is a $yyr$ triangle. 
Since vertices of any yellow edge are endpoints of a red path of length $2$, we see that a $yyr$ triangle implies the existence of an $rrr$ triangle, or a red cycle of length $5$.
Note that $S_{yb}= \alpha(G) = g(n)$. Consider a largest yellow/red clique $X$. We see that since the yellow color class is induced $P_3$-free in $X$, the yellow edges form disjoint cliques in $X$, and there are at most two of them since there are no red triangles. So, $X$ contains a yellow clique on at least $|X|/2$ vertices, i.e., $|X|/2 \leq \alpha(G) =g(n)$, thus $S_{ry}=|X|\leq 2 g(n)$.  Similarly, let $Y$ be a largest red/blue clique. The red edges in it must form a matching, so there is a blue clique of size at least $|Y|/2 $, in particular $|Y|/2 \leq \alpha(G) = g(n)$. Thus $S_{rb} = |Y|\leq 2g(n).$

For the lower bound, consider an $\{rrr, rrb, yyr\}$-avoiding coloring of a complete graph on $n$ vertices.
The red graph $G$ does not have $5$-cycles because otherwise all other edges induced by the vertices of that cycle must be yellow, forcing a $yyr$ triangle. Thus the red graph has odd girth at least $7$.  We have that $S_{by} = \alpha(G) \geq g(n)$.
\end{proof}

\begin{lemma}\label{lem3.rrr.rrb.yyr}
Let $\h = \{\rrr,\rrb,\yyr\}$. Then $h_2(n, \h) = \Omega(n^{2/3}\log^{1/3}n)$ and $h_2(n, \h) = O(n^{3/4}\log n)$. 
\end{lemma}
\begin{proof}
By \Cref{lem:g(n)}  it is sufficient to bound $g(n)$. 

By Caro et al. \cite{CL}, we have $R(C_5, K_t) \le C\frac{t^{3/2}}{\sqrt{\log t}}$, i.e., any $C_5$-free graph on $n$ vertices has independence number at least $C' n^{2/3}\log^{1/3}n$. Thus $g(n) \ge C' n^{2/3}\log^{1/3}n$. 

By  a  result by Spencer \cite{Sp} we have $R(\{C_3, C_4,C_5\}, K_t) \ge C \left({t}/{\log t}\right)^{4/3}$ for some positive constant $C$, i.e., for $n$ sufficiently large there exists a graph  $G$ on $n$ vertices with no $C_3, C_4$, $C_5$ and $\alpha(G) \le C' n^{3/4}\log n $.  Thus  $g(n) \le C'' n^{3/4}\log n$. 
\end{proof}

\begin{lemma}\label{lem3.rrr.rrb.yyb}
Let $\h = \{\rrr,\rrb,\yyb\}$. Then  we have $$h_2(n, \h) = \Omega(n^{2/3}/\log^{1/3} n) \text{ and }\  h_2(n, \h) = O(n^{2/3}\sqrt{\log n}).$$ 
\end{lemma}
\begin{proof}
For the lower bound, let $c$ be an $\h$-avoiding coloring of $K_n$, and let the vertex set be $V$. \\

We shall argue that either our lower bound holds or there is a subset of at least $n/4$ vertices that is a pairwise disjoint union of red/yellow cliques with only blue edges in between. We shall conclude by showing that such a coloring has a large $2$-colored clique of a desired size.\\

Consider a maximal blue clique $B$. 
Then each vertex in $V-B$ sends either a red or a yellow edge to $B$. Moreover, each vertex in $V-B$ sends at most one red and at most yellow edge to $B$. Let $V_1$ be the set of vertices in $V-B$ that send red edges to $B$.  Then these edges form a family $Q$ of pairwise vertex-disjoint stars with centers in $B$. Note that there is no red odd cycle in $V_1$. Indeed, otherwise this cycle contains vertices from three distinct stars from $Q$. In particular, there is a vertex $v$  in one star from $Q$ in $V_1$ that sends red edges in $V_1$  to two distinct stars from $Q$, say with centers $w_1, w_2\in B$.  Then $w_1v$ and $w_2v$ are yellow, so $vw_1w_2$ is a $yyb$ triangle, a contradiction.
So, if $V_1$ has $Cn$ vertices then the red graph induced by $V_1$ has an independent set of size at least $Cn/2$. This gives $S_{yb}\geq Cn/2$. So, we can assume that $|V_1|<Cn$.
We can also assume that $|B|<Cn$. So, let $V_2= V-(B\cup V_1)$. Thus $|V_2|> n/2$ (by taking $C=1/4$).\\

Now, there are only yellow and blue edges between $V_2$ and $B$ and moreover the yellow edges among those form pairwise vertex-disjoint stars with centers in $B$. Let $R_1, \ldots, R_m$ be the intersections of the vertex sets of these stars and  $V_2$. In particular, $V_2$  is the union of the $R_i$s.  Note that there are no yellow edges between $R_i$’s and there are no blue edges within $R_i$’s, otherwise we obtain a $byy$ triangle.\\

There is no vertex $v \in V_2$ that sends a red edge to two different $R_i$'s: 
assume the contrary, i.e., we have red edges $w_1v$ and $w_2v$ with $w_1$, $w_2$ belonging to different $R_i$'s. But then $w_1w_2$ connects two different $R_i$'s, but can be neither red nor blue without creating a forbidden pattern, a contradiction. Thus, the red graph whose edges have endpoints in different $R_i$'s is bipartite. Let $V_3 \subseteq V_2$ be a larger part of such a bipartition (i.e. $|V_3| \ge n/4$) and let $T_i = R_i \cap V_3$.  Then the $T_i$'s  are red/yellow and all edges in between are blue. \\

The remaining part of the proof shows that in any coloring $c'$ of $K_n$  that is formed by pairwise vertex-disjoint red/yellow cliques $T_1, \ldots, T_m$  with all edges between them blue has a large bi-colored clique. This will imply the lower bound $h_2(c) = \Omega(n^{2/3}/(\log^{1/3} n))$. The logarithmic factors here could probably be improved.\\

We shall split the $T_i$'s  according to sizes. Let $X_i$ be the union of all $T_j$ of sizes from $2^i$ to $2^{i+1}-1$, where $i=0,\ldots,\log n$. Consider a largest $X_i$, i.e., $|X_i| \ge |V_3|/\log n \ge  {n}/{(4\log n)}$. Delete at most half the vertices from each $T_j$  in $X_i$ such that all members $T_j$ of $X_i$ have the same size, say $k$. Let the resulting set be $X'$. Let $c'$ be the coloring that results from restricting $c$ to $X'$. Then $c'$ consists of red/yellow cliques of the same size $k$ and only blue edges in between. In addition,  $|X'| \ge {n}/{(8 \log n)}$. 
Then we have $S_{ry} \ge k$. Moreover, since the red graph is triangle-free, \Cref{cor:AKS} implies that we may find a yellow clique of size $C\sqrt{k\log k}$ inside each of the red/yellow cliques. Hence, $$S_{by} \ge C\sqrt{k \log k} \frac{|X'|}{k} \ge C'\frac{n}{\log n} \sqrt{\frac{\log k}{k}}.$$ If $k\le n^{2/3}/\log^{1/3} n $, then $S_{by} \ge C''n^{2/3}/\log^{1/3} n$. Otherwise, we get a large red/yellow clique, which concludes the proof. \\

The upper bound follows from Construction \ref{const9}.
\end{proof}

\begin{lemma}\label{lem3.rrr.bbr.bby}
Let $\h = \{\rrr,\bbr,\bby \}$. Then we have $h_2(n, \h) = \Theta(\sqrt{n \log n})$.
\end{lemma}
\begin{proof}
For the lower bound by \Cref{lem1.rrr} we have $h_2(n, \h) \ge h_2(n, \{rrr\}) = \Theta(\sqrt{n \log n})$.

The upper bound follows from Construction \ref{const4}.
\end{proof}

\begin{lemma}\label{lem3.rrr.bbr.yyr}
Let $\h = \{\rrr,\bbr,\yyr \}$. Then we have $h_2(n, \h) =\ceil{\frac n2} $.
\end{lemma}
\begin{proof}
For the lower bound,  let $c$ be an $\h$-avoiding coloring of $K_n$. 
Assume that $c$ contains a red odd cycle. Let $v_1\cdots v_k v_1$, $k \ge 5$ be a shortest odd red cycle. First let $k \ge 7$. Without loss of generality, we have $c(v_1v_3)=b$. Then $c(v_1v_4)=y$, since otherwise $v_1,v_3,v_4$ induce a $bbr$ triangle.
Thus, we have $c(v_1v_i) = b$ for $i$ odd and $c(v_1v_i) = y$ for $i$ even, and thus, $c(v_1v_{k-1}) = y$.
Since $c(v_1v_3)=b$, we also have $c(v_3v_k) = y$, and hence $c(v_3v_{k-1}) = b$.
Thus, we must have $c(v_{k-1}v_2) = y$, since otherwise $v_2,v_3,v_{k-1}$ induce  a $bbr$ triangle, but then $v_1,v_2,v_{k-1}$ induce a $yyr$ triangle, a contradiction. \\
If $k = 5$, w.l.o.g. we have $c(v_1v_3) = b$. then we must have $c(v_1v_4) = y = c(v_3v_5)$, since otherwise $v_1, v_3, v_4$ or $v_1, v_3, v_5$ induce a $bbr$ triangle. But then we must have $c(v_2v_4) = b = c(v_2v_5)$, since otherwise $v_1, v_2, v_4$ or $v_2, v_3, v_5$ induce a $yyr$ triangle. But then $v_2,v_4,v_5$ induce a $bbr$ triangle, a contradiction.\\
Thus, the red graph is bipartite, so $c$ contains a blue/yellow clique  of size at least $\ceil{\frac n2}$.\\

The upper bound follows from  Construction \ref{const2}.
\end{proof}

\begin{lemma}\label{lem3.rrr.bbr.yyb}
Let $\h = \{\rrr,\bbr,\yyb \}$. Then we have $h_2(n, \h) = \ceil{\frac {3n}7} + \epsilon_1(n)$, where $\epsilon_1(n) = 1$ if $n \equiv 2 \pmod{7}$, and $\epsilon_1(n) = 0$ otherwise.
\end{lemma}
\begin{proof}
The upper bound follows from Construction \ref{const:3n/7}.\\

For the lower bound, let $c$ be an $\h$-avoiding coloring of a $n$-vertex graph on a vertex set $V$.    Assume first that there is a blue clique $B$ of size $|B|\ge 3$. Then each vertex not in $B$ sends at most one yellow edge  to $B$ (otherwise we have a $yyb$ triangle) and no vertex not in $B$ sends both blue and red edges to $B$ (otherwise we have a $bbr$ triangle). Let $A$ be the set of vertices from $V-B$  that send only blue and yellow edges to $B$ and let  $O= V-B-A$, i.e., each vertex from $O$ sends a red edge to $B$. Then, in particular, each vertex from $O$ sends no blue edges to $B$, thus must send at least two red edges and at most one yellow edge to $B$.  
Thus for any two vertices of $A$ there is a vertex in $B$ joined to both of them with blue edges. Similarly, for any two vertices of $O$ there is a vertex in $B$ joined to both of them with red edges. 
Then $A$ induces no red edges and neither does $O$. Note that $B, A$ and $O$ span the whole graph;  $B \cup A$ induces no red edge and $O$ contains no red edge.  Then  $S_{by}^c \geq \max\{|B\cup A|, |O|\}\geq \ceil{\frac n2}$. \\

Thus, we can assume that both the red and blue graph are triangle-free. We shall show that the blue graph has a special structure, in particular it must be a blow-up of a cycle on $7$ vertices.\\

If the blue graph contains no odd cycle, then it is bipartite and hence $S_{ry} \geq \ceil{\frac n2}$.  Thus we can assume that there is a blue odd cycle. The blue graph also does not contain a cycle length $5$ since otherwise all other edges spanned by the vertices of this cycle are yellow, producing a $yyb$ triangle. Assume the shortest blue odd cycle has length  $k\geq 9$. Let $C=C_{k} = v_1\ldots v_{k}v_1$ be a shortest blue odd cycle with $k\ge 9$. Then $C$ has no blue chord, otherwise there  is a shorter blue odd cycle.  Fix a vertex, $v_1$, and  order chords incident to $v_1$ as they appear on the cycle, i.e., $v_1v_3, v_1v_4, \ldots, v_1v_{k-2}$.   Each  chord is red or yellow. There are no two consecutive yellow chords, otherwise we have a $yyb$ triangle. 
 We have that $v_1v_3$ and $v_1v_{k-2}$ are yellow, otherwise there is a $bbr$ triangle.  In addition $v_1v_4$ is red, otherwise there are two consecutive yellow chords. 
 Assume that there are two consecutive red chords incident to $v_1$. Let these chords be, without loss of generality,  $v_1v_i$ and $v_1v_{i+1}$, for  $i>4$. 
Then $v_4v_i$ and $v_4v_{i+1}$ must be yellow, resulting in a $yyb$ triangle.  Thus the chords incident to $v_1$ must have alternating colors $yry \ldots ry$. However, this is impossible since the number of such chords is even.

Thus,  a shortest  blue odd cycle has length $7$. Let the vertex set of such a $7$-cycle be $Y = \{v_1,\ldots, v_7\}$. Then $c(v_iv_j)= y$ if $v_i$ and $v_j$ are at distance two on  $C$ and $c(v_iv_j)= r$ if $v_i$ and $v_j$ are at distance three on $C$.\\

Let $x\in V-Y$. Note that $x$ can send at most $3$ yellow edges to $Y$, otherwise we have a $yyb$ triangle.
Assume $x$ sends no blue edge to $Y$. Then it sends at least four red edges to $Y$, whose endpoints 
contain two vertices at distance $3$ on $C$, which creates an $rrr$ triangle, a contradiction.  Thus, $x$ sends at least one blue edge to $Y$.
Assume $x$ sends exactly one blue edge to $Y$, say to $v_1$. Then $xv_2$ and $xv_7$ must be yellow. Then $xv_3$ and $xv_6$ must be red. But this implies that $xv_3v_6$ forms an $rrr$ triangle, a contradiction.
Note that $x$ sends at most two blue edges to $C$, otherwise there is a shorter blue odd cycle.  Since $x$ sends at least one, at most two, and not exactly one blue edge to $Y$, $x$ sends exactly two blue edges to $Y$.  The endpoints of these two edges must be at distance $2$ on the cycle,  otherwise there is a shorter odd cycle. Without loss of generality, let these endpoints be   $v_2,v_7$. Consider a blue cycle $xv_2v_3\cdots v_7x$.  It must have the same color structure as $C$, i.e., in particular $x$ ``mimics"  $v_1$, i.e., $c(xv_i)=c(v_1v_i)$ for all $v_i \in Y-\{v_1\}$. 
Since $x$ was chosen arbitrarily outside of any blue cycle, we have that each vertex in $V-Y$ ``mimics" some vertex on $C$ and thus the coloring contains a spanning blowup of a coloring $c$ restricted to $Y$.\\

 More specifically, we have that $V$ is a disjoint union of parts  $V_0, \ldots, V_6$ such that for any $i, j \in \{0, 1, \ldots, 6\}$, 
with $i, j$, and any $x\in V_i, z\in V_j$, $c(xz) = b$ if $|i-j|=  1 \pmod 7$, $c(xz) = y$ if $|i-j|=  2 \pmod 7$, $c(xz) = r$ if $|i-j|=  3 \pmod 7$.  All the edges induced by each $V_i$  are yellow. That is, we have  a coloring with a structure as in Construction \ref{const:3n/7} (where the $V_i$'s have variable sizes).\\

For each $i= 0, \ldots, 6$,  let $U_i= V_i\cup V_{i+1}\cup V_{i+2}$ and $W_i = V_i \cup V_{i+2} \cup V_{i+4}$ (indices computed modulo $7$). Then each $U_i$  spans a blue/yellow clique and each $W_i$ spans a red/yellow clique.
We have that $|W_0|+\cdots +|W_6| = 3n$, thus there is $W_i$ of size at least $\ceil{ {3n/7}}$.
This proves the lower bound, except when $n \equiv 2 \pmod{7}$.  The remainder of the proof for this special case can be found in the Appendix.
\end{proof}

\begin{lemma}\label{lem3.rrr.bby.yyb}
	Let $\h = \{\rrr,\bby,\yyb \}$. Then we have $h_2(n, \h) = \ceil{\frac{n}2}$.
\end{lemma}
\begin{proof}
For the lower bound  let $c$ be an $\h$-avoiding coloring. Let $v$ be an arbitrary vertex and denote by $N_r$, $N_b$, $N_y$ its red, blue and yellow neighborhoods, respectively. 

Then $N_r$ does not contain a red edge, so it must induce a monochromatic clique that is either blue or yellow. Without loss of generality, we assume it is blue (a symmetric argument deals in the case when it is yellow).

Between $N_b$ and $N_y$ there cannot be a single blue or yellow edge, so between them we have a complete bipartite red graph. Thus, $N_b$ and $N_y$ cannot induce red edges, so $N_b$ induces  a monochromatic blue and $N_y$ a monochromatic yellow clique. 
Now consider the bipartite graph between $N_r$ and $N_y$. There can be no incident blue and yellow edges and each vertex in $N_y$ sends at most one yellow edge to $N_r$ and each vertex in $N_r$ sends at most one blue edge to $N_y$. Likewise 
 in the bipartite graph between $N_b$ and $N_r$, the yellow edges form a matching and no vertex is incident to both blue and yellow. 

If one of the sets $N_r$, $N_b$, $N_y$ contains at most $1$ vertex, then the larger of the other two sets together with $v$ is a $2$-colored a clique of size at least $\ceil{n/2}$, as required.

So we may assume that each of the sets $N_r, N_b, N_y$ contains at least 2 vertices. We consider two cases:\\

Case $1$: There is a vertex $w^* \in N_r$ that sends only yellow edges to $N_y$.

Every vertex in $N_r^* := N_r - \{w^*\}$ sends only red edges to $N_y$ (otherwise, we obtain either a $yyb$ or $bby$ triangle).  This implies that there are no red edges between $N_r^*$ and $N_b$: if otherwise, we obtain a monochromatic red triangle with vertices in $N_r^*$, $N_b$, and $N_y$, using the fact that the bipartite graph with parts $N_b$ and $N_y$ is entirely red. Moreover, since the yellow edges between $N_r^*$ and $N_b$ form a matching, we have that all edges between $N_r^*$ and $N_b$ must be blue. Now, consider the sets $V_1 = N_r^* \cup N_b \cup \{v\}$ and $V_2 = \{w^*\} \cup N_y \cup \{v\}$. Note that $V_1$ is a red/blue clique and $V_2$ is a red/yellow clique. One of them must have size at least $\ceil{n/2}$, completing the proof in this case.\\

Case $2$: No vertex in $N_r$ sends only yellow edges to $N_y$.

Since no blue and yellow edges are incident in the bipartite graph between $N_r$ and $N_y$, every vertex in $N_r$ sends at least one red edge to $N_y$. Now the proof is similar to the previous case: we have that no edges between $N_r$ and $N_b$ are red (otherwise, we obtain a monochromatic red triangle). Hence, as before all edges between $N_r$ and $N_b$ are blue. So $N_r \cup N_b \cup \{v\}$ is a red/blue clique and $\{v\} \cup N_y$ is a yellow clique. One of these two sets has size at least $\ceil{n/2}$, and this completes the proof of the lower bound.\\

The upper bound follows from Construction \ref{const2} with red and yellow swapped.
\end{proof}

\subsection{$\mathcal H$ contains a monochromatic and a rainbow pattern}\label{size3mono-rb}

\begin{lemma}\label{lem3.rby.rrr.bbb}
Let $\h=\{\rby, \rrr, \bbb\}$. Then $\Omega(n^{3/4}/\log^{3/2} n) = h_2(n, \h) = O(n^{3/4}\sqrt{\log n})$.
\end{lemma}
\begin{proof}
For the lower bound, consider an $\h$-avoiding coloring $c$ of $K_n$. The structure of $c$ is very similar to the structure of  $\{rrr, rby\}$-avoiding colorings.
Consider red components and call their vertex sets \emph{blobs}. Assume that each blob has at least three vertices.  This assumption can be done since if there are at least $n/2$ vertices spanned by red components on at most two vertices, these vertices contain a blue/yellow clique on at least $n/4$ vertices.
Since the coloring is Gallai, \Cref{thm:G} implies that each blob is a union of sets (which we refer to as \emph{sub-blobs}) with all edges between any two sub-blobs of the same color and such that the set of colors between all sub-blobs is either $\{r, b\}$ or $\{r, y\}$.  Moreover, each  sub-blob sends only red edges to some other sub-blob of its blob. All edges between any two blobs are of the same color, either blue or yellow, otherwise there is a rainbow triangle. Lastly, there are no red edges contained in any sub-blob, because otherwise we obtain a monochromatic red triangle.  \\

As in the lemma on $\{rrr, rby\}$-avoiding colorings, (\Cref{lem2.6}) we can assume that for some constant $C''$ there is a subset of  $n' \geq  C''n/\log^2 n$ vertices such that all sub-blobs have the same size, $k$,  and all blobs contain the same number, $\ell$, of sub-blobs. Assume first that at least half the vertices are spanned by blobs with blue/red between sub-blobs.
Then, since the red and blue graph are both triangle-free, there are at most $5$ sub-blobs in each blob, and each sub-blob is blue/yellow. By taking a largest sub-blob in each such blob, we see that $S_{by} \geq n'/10 = \Omega( n/ \log^2 n)$.\\

Now, assume that at least half the vertices are spanned by blobs with red/yellow between sub-blob. All edges between blobs are yellow or blue and all sub-blobs are blue/yellow. As the red graph is triangle-free, applying \Cref{cor:AKS} yields that there are at least $C\sqrt{\ell \log \ell}$ sub-blobs in each blob, such that there are only yellow edges between them. By taking the union of these sets over all blobs we have that 
$$S_{by} \geq C\frac{n'}{k\ell} \sqrt{\ell \log \ell}\cdot k= C\frac{n' \sqrt{\log \ell}}{\sqrt{\ell}}.$$
If $\ell < \sqrt{n}$ then $S_{by} = \Omega(n^{3/4}/ \log^{3/2} n)$.
Otherwise, $\ell \ge \sqrt{n}$.  By picking a yellow clique from each sub-blob and selecting a set of blobs with only yellow edges between them (using \Cref{cor:AKS} again, and the fact that the red and blue graphs are both triangle-free), we have 
$$S_{ry} \geq C'\cdot \ell \sqrt{k \log k}  \sqrt {n'/(k\ell) \log (n'/(k\ell))} = \Omega(\sqrt{\ell n'}) = \Omega(n^{3/4}/\log n).$$

The upper bound follows from Construction \ref{const:3/4_no_rb}.
\end{proof}

\begin{lemma}\label{lem3.rby.rrr.rest}
Let $\h\in \{\{\rby, \rrr, \bbr\}, \{\rby, \rrr, \bby\}, \{\rby, \rrr, \rrb\}\}$.
Then $ h_2(n, \h) = \Omega(n^{2/3} /\log^{3/2}n)$ and $ h_2(n, \h)=  O(n^{2/3} \sqrt{\log n})$. 
\end{lemma}
\begin{proof}
The lower bound follows from \Cref{lem2.6}: $h_2(n, \h) \geq h_2(n, \{rrr, rby\}) = \Omega(n^{2/3}/\log^{3/2} n)$.\\ 

For the upper bound in case $\h\in  \{\{rby, rrr, bbr\}, \{rby, rrr, bby\}\}$, we use Construction \ref{const7} with blue and yellow swapped.
For the upper bound in case $\h=\{rby, rrr, rrb\}$ we use Construction \ref{const9}.
\end{proof}

\section{Final Remarks} \label{Conclusions}

We have determined $h_2(n, \mathcal{H})$ asymptotically up to logarithmic factors for nearly all families $\mathcal{H}$ of at most three patterns. Aside from improving logarithmic factors, there are two major gaps left. First, for the family $\mathcal{H}_0 := \{\rrr, \rrb, \bby\}$ we were able to show that (see \Cref{lem3.rrr.rrb.bby})
\[
\Omega(\sqrt{n\log n}) = h_2(n, \mathcal{H}_0) = O(n^{2/3}\sqrt{\log n}).
\]

We believe that the upper bound is the correct answer (up to logarithmic terms).

Our second gap comes from the family $\mathcal{H}_1 := \{\rrr, \rrb, \yyr\}$. We showed that $h_2(n, \mathcal{H}_1)$ is related to the function $g(n)$ defined as the smallest independence number of an $n$-vertex graph of odd-girth at least $7$:
\[
g(n) = \min\{\alpha(G): |G|= n \text{ and } G \text{ is } \{C_3, C_5\}\text{-free}\}.
\]
In particular, we showed that $g(n) \le h_2(n, \mathcal{H}_1) \le 2g(n)$. It follows that good bounds on the Ramsey number $R(\{C_3, C_5\}, K_n)$ translate to good bounds on $h_2(n, \mathcal{H}_1)$. Using known results on the Ramsey numbers $R(C_5, K_n)$ and $R(\{C_3, C_4, C_5\}, K_n)$, by \Cref{lem3.rrr.rrb.yyr} we have 
\[
\Omega(n^{2/3}\log^{1/3}n) = h_2(n, \mathcal{H}_1) = O(n^{3/4}\log n).
\]
A consequence of the work of Bohman and Keevash~\cite{BK} and Warnke~\cite{W} is that the $C_5$-free process with high probability terminates in a graph whose independence number is $\Theta(n^{3/4}\log^{3/4}n)$. We suspect that the behavior of the independence number does not change much if one forbids triangles in addition to $C_5$'s. Thus, we conjecture that our upper bound on $h_2(n, \mathcal{H}_1)$ is close to the truth:

\begin{conjecture} $$g(n) =  \Omega(n^{3/4})$$ and thus  \[h_2(n, \{\rrr, \rrb, \yyr\}) = \Omega(n^{3/4}). \]
\end{conjecture}
From our reduction to the function $g(n)$, this would follow from the corresponding upper bound on $R(\{C_3, C_5\}, K_n)$. This is likely to be challenging, however, as cycle-complete Ramsey numbers are widely open when the cycle lengths are small and fixed.

Lastly, we decided to stop our investigation at three forbidden patterns. Forbidding more patterns in many cases makes the problem of finding large $2$-colored cliques simple. For this reason (and also for the sake of brevity) we did not pursue this line further. Still, one of course may consider families of forbidden patterns of size four and larger.

\section*{Appendix}

\begin{proof}[Proof of \Cref{prop:2}]

We split the cases according to rainbow and monochromatic patterns:

\begin{itemize}\item{}  $\mathcal H$ contains no rainbow and no monochromatic pattern.  \\
Case 1: the majority color is the same, i.e., w.l.o.g. \rrb\rry.\\
Case 2: the majority color is different, say red and blue. 

Case 2.1: non-majority colors are both not yellow, i.e. \rrb\bbr.

Case 2.2: yellow is a non-majority color in one pattern, i.e. \rrb\bby.

Case 2.3: yellow is a non-majority color  in both patterns, i.e. \rry\bby.\\ 
This gives us \rrb\rry, \rrb\bbr, \rrb\bby, \rry\bby.

\item{} $\mathcal H$ contains a rainbow but no monochromatic triangle, w.l.o.g. \rby\rrb.

\item{} $\mathcal H$ contains a monochromatic triangle and no rainbow traingle,  w.l.o.g. \rrr. 
Then the second pattern  is \rrb, \bbr, \bby or \bbb.  

\item{} $\mathcal H$ contains a monochromatic triangle and a rainbow traingle w.l.o.g. \rrr, \rby. 
\end{itemize}\end{proof}

\begin{proof}[Proof of \Cref{prop:3}]
We split the cases according to rainbow and monochromatic patterns:
\begin{itemize}	
\item{} $\mathcal H$  contains no  rainbow and no   monochromatic triangle.\\
Case 1:  exactly two  patterns have the same majority  color - w.l.o.g. \rrb \rry, third pattern has different majority color, say blue.
Then the third  pattern is either \bbr or \bby.\\
Case 2: patterns have different majority color - w.l.o.g,  $rr*, bb*, yy*$.\\
Then either all non-majority colors are distinct, w.l.o.g. \rrb, \bby, \yyr, or  there are only two different non-majority colors, w.l.o.g \rrb, \bbr, \yyr.

\item{}$\mathcal H$ contains a rainbow and a  no monochromatic pattern.
Then, the other two patterns are listed in the first item of the proof of Proposition \ref{prop:2}.

\item{}$\mathcal H$ contains a monochromatic  and no rainbow  pattern.\\
Case 1: There are at least two monochromatic patterns, say we have \rrr  \bbb. Then all the options for the third pattern up to permutation of patterns are \yyy, \rrb, \rry, \yyr.\\
Case 2: There is only one monochromatic pattern, say \rrr and two non-monochromatic patterns. We have the following cases: 

Case 2.1: Both non-monochromatic triangles have majority color red: \rrr \rrb \rry. 

Case 2.2: Exactly one non-monochromatic triangle has majority color red, w.l.o.g. \rrb. Then all the options for the 3rd pattern are \bbr, \bby, \yyb, \yyr.

Case 2.3: None of the non-monochromatic triangles has majority color red. Then they either have the same majority color, w.l.o.g. \bbr \bby, or we have $bb* yy*$. \\
Then either both non-majority colors are red (\bbr \yyr), exactly one non-majority color is red (w.l.o.g. \bbr \yyb) or no non-majority color is red (\bby \yyb).

\item{}$\mathcal H$ contains a rainbow and a monochromatic pattern, w.l.o.g., \rrr and \rby.
Then the third pattern either has a red majority color, or other majority color, say blue. 
This gives the following options for the third pattern:  \rrb, \bbr, \bby, \bbb.
\end{itemize}
This completes the proof of \Cref{prop:3}.
\end{proof}

\begin{proof}[Proof of Claim 1 in \Cref{lem2.5}]
	Consider a permutation $\pi: [N] \rightarrow [N]$ chosen uniformly at random and apply it to $G$. Let $E_\pi = \{\pi(u)\pi(v): uv \in E(G)\}$. For each edge $e \in E(G)$ we say that $e$ is \emph{bad} if $e \in E_\pi$, and we let $X$ be the random variable counting bad edges. For each edge $e = uv \in E(G)$ there are $2|E(G)|(N-2)!$ permutations that can make $e$ bad. Hence
	\[
	\mathbb{P}(e \text{ is bad}) = \frac{2|E(G)|(N-2)!}{N!} = \frac{|E(G)|}{\binom{N}{2}}.
	\]
	Thus, $\mathbb{E}[X] = |E(G)|^2/ \binom{N}{2}$, and so there is a permutation $\sigma$ such that
	\[
	|E(G)\cap E_\sigma| \le \frac{|E(G)|^2}{\binom{N}{2}}.
	\]
\end{proof}

\begin{proof}[Proof of  \Cref{lem:rrr.rrb.bby-structure}]
	Let $c$ be an $\h$-avoiding coloring of $K_n$. 
	Start by partitioning the vertex set into blue/red cliques by greedily picking a maximal red/blue clique at each step. Let $\C$ be the set of these cliques. Note that there is a yellow edge between any two cliques from $\C$. Within each clique the red graph is a matching and the red edges between any two components also form a matching (otherwise we have $rrb$ or $rrr$ triangles).
	Let $\C'$ be the set of cliques from $\C$ on at least $4$ vertices each.\\
	
	Assume first  that at least $n/2$ vertices are spanned by cliques from $\C'$. 
	We shall show first that there is no blue edge between any two cliques from $\C'$.  Assume  $U,V\in \C'$ with $|U| \ge |V|$ and there is a blue edge $u'v$, $u'\in U$, $v\in V$.
	If there is a yellow edge $u''v$, for some $u''\in U$, then $u'u''$ must be red (since otherwise $u'u''v$ is a $bby$ triangle).   Let $u \in U - \{u',u''\}$.
	Since red forms a matching within $U$, we have $c(u''u)=c(u'u) = b$.
	Then $uv$ cannot be blue ($u''uv$ would induce a  $bby$ triangle), and it cannot be yellow ($u'vu$ would induce a $bby$ triangle), so it must be red. Since this holds for each $u \in U-\{u',u''\}$, we must have $|U| \le 3$ since $v$ cannot have two red neighbours in $U$, a contradiction to our assumption that $|U|\geq 4$. Thus, the assumption that there is a yellow edge is wrong, so all edges from $v$ to $U$ are blue or red. But then $U \cup \{v\}$ would form a larger red/blue clique and contradict how we chose $\C$ greedily. Thus, between any two cliques from $\C'$  there is no blue edge and red forms a matching between any two cliques from $\C'$. Thus our structural result follows by choosing at least half the vertices from each clique of $\C'$  so that these vertices induce a blue clique.\\
	
	Now assume that at least $n/2$ vertices are spanned by cliques from $\C-\C'$, i.e., cliques of size at most $3$ each.  We have that $|\C-\C'| \geq n/6$.  Each of the cliques from $\C-\C'$  either spans a blue triangle or  not.  Let $\C'' \subseteq \C-\C'$ be the set of cliques forming blue triangles. We distinguish the following cases: \\
	
	Case 1:  $|\C''|\geq n/12$. Since there is a yellow edge between any two  cliques from $\C$,  there can't be a blue edge between any two members of $\C''$, otherwise we create a $bby$ triangle. Thus, we can pick one vertex from each member of  $\C''$ and have a red/yellow clique of size at least $n/12 \in \Omega(n^{2/3})$. \\
	
	Case 2:  $|\C''| <n/12$, i.e., at least $n/12$ cliques from $\C-\C'$  do not span a blue triangle. Let $G$ be the subgraph spanned by vertices of $\C-(\C' \cup \C'')$ with the inherited coloring.  Assume that $G$ contains  a blue $C_5$. Then all edges within this cycle must be red (no blue $K_3$, no $bby$ triangle), so we have a red/blue $K_5$, which contains an $rrb$ triangle, a contradiction. Thus, the blue subgraph of $G$ is  $C_5$-free, and since  $R(C_5, K_k) = O(k^{3/2}/\sqrt{\log k})$ we have a red/yellow clique of a desired size.
	
\end{proof}

\begin{proof}[Proof of special case in \Cref{lem3.rrr.bbr.yyb} ]
	Let $n \equiv 2 \pmod{7}$, i.e., $n= 7k+2$ for an integer $k$. 
	We shall show that there is a $2$-colored clique of size at least $\ceil{3n/7}+1$. In order to do so, we first shall show that the sets $V_i$'s pairwise differ in size by  at most $1$.\\

	If there is an index $i'$, such that $|W_{i'}| \ge \ceil{3n/7}+1$, we are done, so assume $|W_i|\le \ceil{3n/7}$ for $i=0,\ldots,6$. 
	Now assume there is an index $i'$ such that $|W_{i'}| \le \floor{3n/7} - 1$. But then we have 
	$\sum_{i=0}^6 |W_i| \le \floor{3n/7} - 1 + 6 \ceil{3n/7} =   \floor{3(7k+2)/7}-1+6\ceil{3(7k+2)/7} = 21k + 5 < 21k+ 6 = 3n$, a contradiction.  Thus $|W_i|$ and $|W_j|$  differ by at most $1$ for $0 \le i< j \le 6$. Similarly, all $U_i$'s differ in size by at most $1$. \\
	
	Note that $W_i \cap W_{i+2} = V_{i+2} \cup  V_{i+4}$ for $i=0,\ldots,6$ (indices computed modulo $7$), so the symmetric difference $W_i\triangle W_{i+2} = V_i \cup V_{i+6}$. By the above observation, that means in particular that $|V_{i-1}|$ and $|V_{i}|$ differ by at most $1$, for each $i = 0, 1, \ldots, 6$.  Similarly, by considering two consecutive $U_i$'s, we see that 
	$|V_i|$ and $|V_{i+3}|$ differ in size by at most $1$. Then it is clear that $||V_0|-|V_2||\leq 2$.  Assume that  $|V_0|=t$ and 
	$|V_2|=t+2$, for some $t$. Then $|V_1|=t+1$ and $|V_3| \geq t+1$. Thus $|U_1|\geq 3t+4$ that implies that $3t+3 \leq |U_6| = |V_6| + t + t+1 $, that in turn implies that $|V_6| \geq t+2$. This contradicts the fact that $|U_6|$ and $|U_0|$ differ by at most $1$. 
	It shows that $||V_i|-|V_{i+2}||\leq 1$.  Together with the fact that $||V_i|-|V_{i+1}||\leq 1$ and $||V_i|-|V_{i+3}||\leq 1$, we see that any two set $V_i, V_j$, $i, j\in \{0, 1, \ldots, 6\}$ differ in size by at most $1$.
	Thus $V_i$'s have sizes either $\ceil{n/7}$ or $\floor{n/7}$.  Since $n = 7k+2$, there are exactly two parts  $V_i, V_j$ of sizes $\ceil{n/7}$.  No matter how they are located, there is a third part $V_g$ such that $V_i\cup V_j\cup V_g$ is either $W_m$ or $U_m$ for some $m$. This gives a two-colored clique on $2\ceil{n/7} + \floor{n/7} = \ceil{3n/7} +1$ vertices. 
\end{proof}

{\bf Acknowledgements} The authors thank two anonymous referees for many excellent comments, for spotting some mistakes, and for  suggestions improving the manuscript.

\end{document}